\documentclass[final]{siamart190516}
\usepackage{amsfonts}
\usepackage{amsfonts,amsmath,amssymb}
\usepackage{mathrsfs,mathtools,stmaryrd,wasysym}
\usepackage{enumerate}
\usepackage{enumitem}
\usepackage{esint}
\usepackage{graphicx,psfrag}
\usepackage{hyperref}
\usepackage{amsmath,stackengine}
\usepackage{yfonts}
\DeclareMathAlphabet{\mathpzc}{OT1}{pzc}{m}{it}

\newcommand{\FF}[1]{{\color{black}#1}}
\newcommand{\EO}[1]{{\color{black}#1}}

\newcommand{\T}{\mathscr{T}}

\newsiamremark{remark}{Remark}
\newcommand{\TheTitle}{Error estimates for a pointwise tracking optimal control problem of a semilinear elliptic equation}
\newcommand{\ShortTitle}{A semilinear pointwise tracking control problem}
\newcommand{\TheAuthors}{A. Allendes, F. Fuica, E. Ot\'arola}

\headers{\ShortTitle}{\TheAuthors}

\title{{\TheTitle}\thanks{AA is partially supported by CONICYT through FONDECYT project 1170579. FF is supported by UTFSM through Beca de Mantenci\'on. EO is partially supported by CONICYT through FONDECYT Project 11180193.}}

\author{Alejandro Allendes\thanks{Departamento de Matem\'atica, Universidad T\'ecnica Federico Santa Mar\'ia, Valpara\'iso, Chile.
(\email{alejandro.allendes@usm.cl}).}
\and
Francisco Fuica\thanks{Departamento de Matem\'atica, Universidad T\'ecnica Federico Santa Mar\'ia, Valpara\'iso, Chile.
(\email{francisco.fuica@sansano.usm.cl}).}
\and
Enrique Ot\'arola\thanks{Departamento de Matem\'atica, Universidad T\'ecnica Federico Santa Mar\'ia, Valpara\'iso, Chile.
(\email{enrique.otarola@usm.cl}, \url{http://eotarola.mat.utfsm.cl/}).}
}

\ifpdf
\hypersetup{
  pdftitle={\TheTitle},
  pdfauthor={\TheAuthors}
}
\fi

\date{Draft version of \today.}

\begin{document}

\maketitle

\begin{abstract}
We consider a pointwise tracking optimal control problem for a semilinear elliptic partial differential equation. We derive the existence of optimal solutions and analyze first and, necessary and sufficient, second order optimality conditions. We devise two strategies of discretization to approximate \EO{a} solution of the optimal control problem: a semidiscrete scheme where the control variable is not discretized --  the so-called variational discretization approach -- and a fully discrete scheme where the control variable is discretized with piecewise constant functions. \EO{For both solution techniques}, we analyze convergence properties of discretizations and derive error estimates.
\end{abstract}



\section{Introduction}\label{sec:intro}
In this work we are interested in the analysis and discretization of a pointwise tracking optimal control problem for a semilinear elliptic partial differential equation (PDE). This PDE-constrained optimization problem entails the minimization of a cost functional that involves point evaluations of the state; control constrains are also considered. To make matters precise, we let $\Omega\subset\mathbb{R}^d$, with $d\in \{2, 3\}$, be an open, bounded, and convex polytope with  boundary $\partial\Omega$ and $\mathcal{D}$ be a finite ordered subset of $\Omega$ with cardinality $\#\mathcal{D}<\infty$. Given a set of desired states $\{y_t\}_{t\in \mathcal{D}}\subset\mathbb{R}$ and a regularization parameter $\alpha>0$, we define the cost functional
\begin{equation}\label{def:cost_func}
J(y,u):=\frac{1}{2}\sum_{t\in \mathcal{D}}\left(y(t)-y_t\right)^{2}+\frac{\alpha}{2}\|u\|_{L^2(\Omega)}^2.
\end{equation} 
The problem under consideration reads as follows: Find $\min J(y,u)$ subject to the \emph{monotone, semilinear, and elliptic PDE}
\begin{equation}\label{def:state_eq}
-\Delta y + a(\cdot,y)  =  u  \text{ in }  \Omega, \qquad
y  =  0  \text{ on }  \partial\Omega,
\end{equation}
and the \emph{control constraints}
\begin{equation}\label{def:box_constraints}
u \in \mathbb{U}_{ad},\qquad \mathbb{U}_{ad}:=\{v \in L^2(\Omega):  \texttt{a} \leq v(x) \leq \texttt{b} \text{ a.e.}~x \in \Omega \}.
\end{equation}
The control bounds $\texttt{a},\texttt{b} \in \mathbb{R}$ are such that $\texttt{a} < \texttt{b}$. Assumptions on the nonlinear function $a$ will be deferred until section \ref{sec:assumption}.

The study of a priori error estimates for finite element approximations of distributed semilinear optimal control problems has previously been considered in a number of works. To the best of our knowledge, the work \cite{MR1937089} appears to be the first to provide error estimates for such a class of problems; control constraints are considered. Within a general setting, the authors consider the cost functional $J(y,u):=\int_{\Omega} L(x,y(x),u(x)) \mathrm{d}x$, where $L$ satisfies the conditions stated in \cite[assumption \textbf{A2}]{MR1937089}, and devise finite element techniques to solve the underlying optimal control problem. To be precise, the authors
propose a fully discrete scheme on quasi-uniform meshes that discretize the control variable with piecewise constant functions and the state and adjoint variables with piecewise linear functions. Assuming that $\Omega \subset \mathbb{R}^d$, with $d \in \{2,3\}$, is a convex domain with a boundary $\partial \Omega$ of class $C^{1,1}$ and that the mesh-size is sufficiently small, the authors derive a priori estimates for the error approximation of \EO{an} optimal control variable in the $L^2(\Omega)$-norm \cite[Theorem 5.1]{MR1937089} and the $L^{\infty}(\Omega)$-norm \cite[Theorem 5.2]{MR1937089}. Since the publication of \cite{MR1937089}, several additional studies have enriched the understanding of error estimates for semilinear PDE-contrained optimization problems. We refer the reader to \cite{MR3586845} for references and also for an up-to-date discussion including linear approximation of the optimal control, the so-called variational discretization approach, superconvergence, post--processing schemes, and time dependent problems.

For the particular case $a \equiv 0$, there are several works available in the literature that provide a priori error estimates for finite element discretizations of \eqref{def:cost_func}--\eqref{def:box_constraints}. In two and three dimensions and utilizing that the associated adjoint variable belongs to $W_0^{1,r}(\Omega)$, for every $r < d/(d-1)$, the authors of \cite{MR3449612} obtain a priori and a posteriori error estimates for the so-called variational discretization approach; the state and adjoint equations are discretized with continuous piecewise linear functions. The a priori estimate for the error approximation of the control variable behaves as $\mathcal{O}(h^{2-d/2})$ \cite[Theorem 3.2]{MR3449612}. Later, the authors of \cite{MR3523574} analyze a fully discrete scheme that approximates the optimal state, adjoint, and control variables with piecewise linear functions and obtain a $\mathcal{O}(h)$ rate of convergence for the error approximation of the control variable in two dimensions \cite[Theorem 5.1]{MR3523574}. The authors also analyze the variational discretization scheme and derive error estimates \cite[Theorem 5.2]{MR3523574}. In \cite{MR3800041}, the authors invoke the theory of Muckenhoupt weights and weighted Sobolev spaces and analyze a numerical scheme that discretizes the control variable with piecewise constant functions;  the state and adjoint equations are 
discretized as above. In two and three dimensions, the authors derive estimates for the error approximation of the optimal control variable; the one in two dimensions being nearly-optimal in terms of approximation \cite[Theorem 4.3]{MR3800041}. In three dimensions the estimate is suboptimal: $\mathcal{O}(h^{1/2}|\log h|)$. This has been recently improved to $\mathcal{O}(h|\log h|)$ in \cite[Theorem 6.6]{MR3973329}. In addition, the authors of \cite{MR3973329} provide, in two dimensions, a $\mathcal{O}(h^2|\log h|^2)$ error estimate for the variational discretization approach \cite[Theorem 7.5]{MR3973329} and a post--processing scheme \cite[Theorem 7.12]{MR3973329}. \EO{We finally mention the works \cite{MR4013930,FOQ,Behringer} for extensions of the aforementioned results to the Stokes system and the work \cite{Otarola_semilinear_deltas} for the analysis and discretization of an optimal control problem for a semilinear PDE with a control variable that corresponds to the amplitude of forces modeled as point sources.} 

In contrast to the aforementioned advances and to the best of our knowledge, this exposition is the first one that studies approximation techniques for a pointwise tracking optimal control problem involving a semilinear elliptic PDE. In what follows we list what, we believe, are the main contributions of our work:
\begin{itemize}
\item[$\bullet$] \emph{Existence of an optimal control:} Assuming that $a = a(x,y)$ is a Carath\'eodory function that is monotone increasing and locally Lipschitz in $y$ with $a(\cdot,0) \in L^2(\Omega)$, we show that our control problem admits a solution (Theorem \ref{thm:existence_control}).
\item[$\bullet$] \emph{Optimality conditions:} Under additional assumptions on $a$, we obtain first order optimality conditions in Theorem \ref{thm:optimality_cond} and second order necessary and sufficient optimality conditions with a minimal gap in section \ref{sec:2nd_order}. Since the cost functional of our problem involves point evaluations of the state, we have that $\bar p \in W_0^{1,r}(\Omega) \setminus (H_0^1(\Omega) \cap C(\bar \Omega))$, where $r<d/(d-1)$. This requires a suitable adaption of the arguments available in the literature \cite[section 6]{MR3586845}, \cite{Troltzsch}: the arguments in \cite[section 6]{MR3586845} utilize that $\bar p \in W^{2,p}(\Omega)$ with $p>d$.
\item[$\bullet$] \emph{Convergence of discretization:}  
We prove the existence of subsequences of global solutions of suitable discrete problems that converge to a global solution of the continuous control problem.  We also prove that continuous  strict local solutions can be approximated by local minima of discrete problems.
\item[$\bullet$] \emph{Error estimates:}  We obtain a nearly--optimal local error estimate in $L^{\infty}$ for semilinear PDEs (Theorem \ref{thm:error_estimates_state}). This estimate is instrumental to derive error estimates for the solution schemes that we propose (Theorems \ref{thm:control_estimate}, \ref{thm:control_estimate_improved}, \ref{thm:control_estimate_var1}, and \ref{thm:control_estimate_var}). The analysis involves estimates in $L^{\infty}$-norm and $W^{1,p}$-spaces, combined with having to deal with first and second order optimality conditions. This subtle intertwining of ideas is one of the highlights of this contribution.
\end{itemize}

The outline of this manuscript is as follows. We set notation and introduce the functional framework we shall work with in section \ref{sec:nota_and_assum}. In section \ref{sec:tracking_ocp}, we analyze a weak version of the optimal control problem \eqref{def:cost_func}--\eqref{def:box_constraints}. In particular, we show existence of solutions and obtain first and second order optimality conditions. Section \ref{sec:fem} presents two finite element discretizations for \eqref{def:cost_func}--\eqref{def:box_constraints} and review some results related to the discretization of the state and adjoint equations. The convergence of discretizations is provided in section \ref{sec:conver_of_disc}. In section \ref{sec:error_estimates}, we derive estimates for the error approximation of an optimal control variable. We conclude in section \ref{sec:num_ex} by providing a numerical example that illustrates the theory.


\section{Notation and assumptions}\label{sec:nota_and_assum}

Let us set notation and describe the setting we shall operate with.


\subsection{Notation}\label{sec:notation}

Throughout this work $d \in \{2,3\}$ and $\Omega\subset\mathbb{R}^d$ is an open, bounded, and convex polytopal domain. If $\mathscr{X}$ and $\mathscr{Y}$ are Banach function spaces, we write $\mathscr{X}\hookrightarrow \mathscr{Y}$ to denote that $\mathscr{X}$ is continuously embedded in $\mathscr{Y}$. We denote by $\| \cdot \|_{\mathscr{X}}$ the norm of $\mathscr{X}$. Given $r \in (1,\infty)$, we denote by $r'$ its H\"older conjugate, i.e., the real number such that $1/r + 1/r' = 1$. The relation $\mathfrak{a} \lesssim \mathfrak{b}$ indicates that $\mathfrak{a} \leq C \mathfrak{b}$, with a positive constant that depends neither on $\mathfrak{a}$, $\mathfrak{b}$ nor on the discretization parameters. The value of $C$ might change at each occurrence.


\subsection{Assumptions}\label{sec:assumption}

We will operate under the following assumptions on the nonlinear function $a$. We notice, however, that some of the results that we present in this work are valid under less restrictive requirements. When possible we explicitly mention the assumptions on $a$ that are needed to obtain a particular result.

\begin{enumerate}[label=(A.\arabic*)]
\item \label{A1} $a:\Omega\times \mathbb{R}\rightarrow  \mathbb{R}$ is a Carath\'eodory function of class $C^2$ with respect to the second variable and $a(\cdot,0)\in L^{2}(\Omega)$.

\item \label{A2} $\frac{\partial a}{\partial y}(x,y)\geq 0$ for a.e.~$x\in\Omega$ and for all $y \in \mathbb{R}$.

\item \label{A3} For all $\mathsf{m}>0$, there exists a positive constant $C_{\mathsf{m}}$ such that 
\begin{equation*}
\sum_{i=1}^{2}\left|\frac{\partial^{i} a}{\partial y^{i} }(x,y)\right|\leq C_{\mathsf{m}}, 
\qquad
\left|\frac{\partial^{2} a}{\partial y^{2} }(x,v)- \frac{\partial^{2} a}{\partial y^{2} }(x,w)\right|\leq C_{\mathsf{m}} |v-w|
\end{equation*}
for a.e.~$x\in \Omega$ and $y,v,w \in [-\mathsf{m},\mathsf{m}]$.
\end{enumerate}


\subsection{State equation}\label{sec:state_equation}
Here, we collect some facts on problem \eqref{def:state_eq} that are well-known and will be used repeatedly. For $f \in L^q(\Omega)$, with $q > d/2$, we introduce the following weak formulation of problem \eqref{def:state_eq}:
Find $\mathsf{y}\in H_0^1(\Omega)$ such that
\begin{equation}\label{eq:weak_semilinear_pde}
(\nabla \mathsf{y},\nabla v)_{L^2(\Omega)}+(a(\cdot,\mathsf{y}),v)_{L^2(\Omega)}=(f,v)_{L^2(\Omega)} \quad \forall v\in H_0^1(\Omega).
\end{equation}

We begin with the following result that states the well-posedness of problem \eqref{eq:weak_semilinear_pde} and further regularity properties for its solution $\mathsf{y}$.

\begin{theorem}[well--posedness and regularity]\label{thm:well_posedness_semilinear}
Let $f \in L^q(\Omega)$ with $q>d/2$. Let $a = a(x,y) : \Omega \times \mathbb{R} \rightarrow \mathbb{R}$ be a Carath\'eodory function that is monotone increasing and locally Lipschitz in $y$ a.e.~in $\Omega$. If $\Omega$ denotes an open and bounded domain with Lipschitz boundary and $a(\cdot,0) \in L^q(\Omega)$, with $q>d/2$, then problem \eqref{eq:weak_semilinear_pde} has a unique solution $\mathsf{y} \in H_0^1(\Omega) \cap L^{\infty}(\Omega)$. If, in addition, $\Omega$ is convex and $f,a(\cdot,0) \in L^2(\Omega)$, then
$
\| \mathsf{y} \|_{H^{2}(\Omega)}
\lesssim
\|f-a(\cdot,0)\|_{L^{2}(\Omega)},
$
with a hidden constant that is independent of $a$ and $f$.
\end{theorem}
\begin{proof}
The existence of a unique solution $\mathsf{y} \in H_0^1(\Omega)\cap L^\infty(\Omega)$ follows from the main theorem on monotone operators \cite[Theorem 26.A]{MR1033498}, \cite[Theorem 2.18]{MR3014456} combined with an argument due to Stampacchia \cite{MR192177}, \cite[Theorem B.2]{MR1786735}. The $H^2(\Omega)$-regularity of $\mathsf{y}$ follows from the fact that $f,a(\cdot, 0) \in L^2(\Omega)$ and that $\Omega$ is convex; see \cite[Theorems 3.2.1.2 and 4.3.1.4]{MR775683} for $d=2$ and \cite[Theorems 3.2.1.2]{MR775683} and \cite[section 4.3.1]{MR2641539} for $d=3$.
\end{proof}

We conclude the section with the following Lipschitz property.

\begin{theorem}[Lipschitz property]
\label{thm:lipschitz_property}
Let $\Omega$ be an open and bounded domain with Lipschitz boundary and let $f_{1},f_{2} \in L^q(\Omega)$ with $q>d/2$. Let $a = a(x,y) : \Omega \times \mathbb{R} \rightarrow \mathbb{R}$ be a Carath\'eodory function of class $C^1$ with respect to $y$ such that \textnormal{\ref{A2}} holds. Assume that, for every $\mathsf{m}>0$, $|\partial a / \partial y (x,y)| \leq C_{\mathsf{m}}$ for a.e.~$x \in \Omega$ and $y \in [-\mathsf{m},\mathsf{m}]$. If $a(\cdot,0) \in L^q(\Omega)$, with $q>d/2$, then 
\begin{equation}\label{eq:lipschitz_prop}
\| \nabla (\mathsf{y}_{1}-\mathsf{y}_{2} ) \|_{L^2(\Omega)}
+
\|\mathsf{y}_{1}-\mathsf{y}_{2}\|_{L^\infty(\Omega)} 
\lesssim
\|f_{1}-f_{2}\|_{L^{q}(\Omega)},
\end{equation}
where $\mathsf{y}_{i} \in H_0^1(\Omega)\cap L^\infty(\Omega)$, with $i \in \{1,2\}$, solves  \eqref{eq:weak_semilinear_pde} with $f$ replaced by $f_{i}$.
\end{theorem}
\begin{proof}
Estimate \eqref{eq:lipschitz_prop} follows from the arguments developed in the proof of \cite[Theorem 4.16]{Troltzsch} upon suitable modifying boundary conditions and invoking \cite[Theorem B.2]{MR1786735}; observe that the last observation in \cite[Appendix B]{MR1786735} guarantees that \cite[Theorem B.2]{MR1786735} holds for a forcing term in $L^q(\Omega)$ with $q>d/2$.
\end{proof}


\section{The pointwise tracking optimal control problem}
\label{sec:tracking_ocp}

In this section, we analyze the following weak version of problem \eqref{def:cost_func}--\eqref{def:box_constraints}: Find
\begin{equation}\label{def:weak_ocp}
\min \{ J(y,u): (y,u) \in H_0^1(\Omega)\cap L^\infty(\Omega) \times \mathbb{U}_{ad}\}
\end{equation}
subject to the monotone, semilinear and elliptic \emph{state equation}
\begin{equation}\label{eq:weak_st_eq}
(\nabla y,\nabla v)_{L^2(\Omega)}+(a(\cdot,y),v)_{L^2(\Omega)}=(u,v)_{L^2(\Omega)} \quad \forall v \in H_0^1(\Omega).
\end{equation}

Let $a = a(x,y) : \Omega \times \mathbb{R} \rightarrow \mathbb{R}$ be a Carath\'eodory function that is monotone increasing and locally Lipschitz in $y$ with $a(\cdot,0) \in L^2(\Omega)$. Since $\Omega$ is convex, Theorem \ref{thm:well_posedness_semilinear} yields the existence of a unique solution $y \in H^2(\Omega) \cap H_0^1(\Omega)$ of problem \eqref{eq:weak_st_eq}. We immediately notice that, in view of the continuous embedding $H^2(\Omega) \hookrightarrow C(\bar \Omega)$, which holds because $d \in \{2,3\}$, point evaluations of $y$ in \eqref{def:cost_func} are well-defined.


\subsection{Existence of optimal controls}
Let us introduce the control to state map $\mathcal{S}:L^2(\Omega) \rightarrow H_0^1(\Omega) \cap H^2(\Omega)$, which given a control $u$ associates to it the state $y$ that solves \eqref{eq:weak_st_eq}, and define $j:L^2(\Omega) \rightarrow \mathbb{R}$ by $j(u):= J(\mathcal{S}u,u)$. With these ingredients at hand, the existence of an optimal state-control pair $(\bar y, \bar u)$ is as follows.

\begin{theorem}[existence of an optimal pair]\label{thm:existence_control}
Let $\Omega$ be an open, bounded, and convex domain. Let $a = a(x,y) : \Omega \times \mathbb{R} \rightarrow \mathbb{R}$ be a Carath\'eodory function that is monotone increasing and locally Lipschitz in $y$ with $a(\cdot,0) \in L^2(\Omega)$. Thus, the optimal control problem \eqref{def:weak_ocp}--\eqref{eq:weak_st_eq} admits at least one solution $(\bar{y},\bar{u}) \in H_{0}^{1}(\Omega) \cap H^2(\Omega) \times \mathbb{U}_{ad}$.
\end{theorem}
\begin{proof}
We follow the proof of \cite[Theorem 4.15]{Troltzsch} and define 
\begin{equation*}
\Phi: L^2(\Omega) \rightarrow \mathbb{R},
\quad
u  \mapsto \tfrac{\alpha}{2} \| u \|^2_{L^2(\Omega)},
\quad
\Psi: H^2(\Omega) \cap H_0^1(\Omega) \rightarrow \mathbb{R},
\quad
y \mapsto \tfrac{1}{2}\sum_{t\in\mathcal{D}} |y(t) - y_t|^2.
\end{equation*}
Since $J(y,u) = \Psi(y) + \Phi(u)$ is bounded from below, the infimum $\mathfrak{j}:=\inf \{ J(\mathcal{S}u,u), u\in\mathbb{U}_{ad} \}$ exists. Let $u_k \in \mathbb{U}_{ad}$ and $y_k=\mathcal{S}u_k$ be such that $J(y_k,u_k) \rightarrow \mathfrak{j}$ as $k \uparrow \infty$, i.e., $\{(y_k,u_k)\}_{k\in\mathbb{N}}$ is a minimizing sequence. Since $\mathbb{U}_{ad}$ is weakly sequentially compact in $L^2(\Omega)$, there exists a nonrelabeled subsequence $\{u_k\}_{k\in\mathbb{N}}$ such that $u_k\rightharpoonup \bar{u}$ in $L^2(\Omega)$ as $k\uparrow \infty$; $\bar{u}\in \mathbb{U}_{ad}$. On the other hand, since $\{ u_k\}_{k \in \mathbb{N}} \subset \mathbb{U}_{ad}$, Theorem \ref{thm:well_posedness_semilinear} yields the existence of $\mathsf{m}>0$ such that $|y_k(x)| \leq \mathsf{m}$ for a.e.~$x \in \Omega$ and $k \in \mathbb{N}$. This implies that $\{ a(\cdot,y_k) \}_{k \in \mathbb{N}}$ is bounded in $L^2(\Omega)$. We can thus conclude the existence of a nonrelabeled subsequence $\{y_k\}_{k\in\mathbb{N}}$ such that $y_k \rightharpoonup \bar{y}$ in $H^2(\Omega)\cap H_0^1(\Omega)$, which, in turns, implies that $y_k\to \bar{y}$ in $H_0^1(\Omega) \cap C(\bar{\Omega})$ as $k\uparrow \infty$; $\bar y$ being the natural candidate for an optimal state. Observe that $\bar y \in L^{\infty}(\Omega)$ and that $\| a(\cdot,y_k) - a(\cdot,\bar y) \|_{L^2(\Omega)} \leq \mathfrak{l}_{\mathsf{m}} \| y_k - \bar y \|_{L^2(\Omega)} \rightarrow 0$ as $k \uparrow \infty$. Following similar arguments to the ones elaborated in \cite[Theorem 4.15]{Troltzsch}, we can thus conclude that $\bar{y}$ solves \eqref{eq:weak_st_eq} with $u=\bar{u}$. Finally, in view of the fact that $\Phi$ is continuous and convex in $L^2(\Omega)$ we deduce that $J(\bar{y},\bar{u})=\mathfrak{j}$.
\end{proof}


\subsection{First order necessary optimality conditions}\label{sec:1st_order}

Since the optimal control problem \eqref{def:weak_ocp}--\eqref{eq:weak_st_eq} is not convex, we discuss optimality conditions in the context of local solutions in $L^2(\Omega)$ \cite[section 4.4.2]{Troltzsch}. In this section, we formulate first order necessary optimality conditions. To accomplish this task, we begin by analyzing differentiability properties of the control to state operator $\mathcal{S}$.

\begin{theorem}[differentiability properties of $\mathcal{S}$]
\label{thm:properties_C_to_S}
Assume that \textnormal{\ref{A1}}, \textnormal{\ref{A2}}, and \textnormal{\ref{A3}} hold. Then, $\mathcal{S}: L^{2}(\Omega) \rightarrow H_0^1(\Omega) \cap H^{2}(\Omega)$ is of class $C^2$. In addition, if $u,v\in L^{2}(\Omega)$, then $z=\mathcal{S}'(u)v\in H_0^1(\Omega) \cap H^{2}(\Omega)$ corresponds to the unique solution to 
\begin{equation}\label{eq:aux_adjoint}
(\nabla z,\nabla w)_{L^2(\Omega)}+\left(\tfrac{\partial a}{\partial y}(\cdot,y)z,w\right)_{L^2(\Omega)}=(v,w)_{L^2(\Omega)} \quad \forall w\in H_0^1(\Omega),
\end{equation}
where $y = \mathcal{S}u$. If $v_1,v_2\in L^{2}(\Omega)$, then $\mathfrak{z}=\mathcal{S}''(u)(v_1,v_2)\in H_0^1(\Omega) \cap H^{2}(\Omega)$ is the unique solution to 
\begin{equation}\label{eq:aux_adjoint_diff_2}
(\nabla \mathfrak{z},\nabla w)_{L^2(\Omega)}+\left(\tfrac{\partial a}{\partial y}(\cdot,y)\mathfrak{z},w\right)_{L^2(\Omega)}=-\left(\tfrac{\partial^2 a}{\partial y^2}(\cdot,y)z_{v_1}z_{v_2},w\right)_{L^2(\Omega)}
\end{equation}
for all $w\in H_0^1(\Omega)$, where $z_{v_i}=\mathcal{S}'(u)v_i$, with $i \in \{ 1,2 \}$, and $y = \mathcal{S}u$.
\end{theorem}
\begin{proof}
The first order Fr\'echet differentiability of $\mathcal{S}$ from $L^2(\Omega)$ into $H_0^1(\Omega) \cap H^2(\Omega)$ follows from a slight modification of the arguments of \cite[Theorem 4.17]{Troltzsch} that basically entails replacing $H^1(\Omega) \cap C(\bar \Omega)$ by $H_0^1(\Omega) \cap H^2(\Omega)$ and $L^r(\Omega)$ by $L^2(\Omega)$. \cite[Theorem 4.17]{Troltzsch} also yields that $z=\mathcal{S}'(u)v\in H^{2}(\Omega)\cap H_0^1(\Omega)$ corresponds to the unique solution to \eqref{eq:aux_adjoint}. The second order Fr\'echet differentiability of $\mathcal{S}$ can be obtained by using the implicit function theorem; see, for instance, the proof of \cite[Theorem 4.24]{Troltzsch} and \cite[Proposition 16]{MR3586845} for details.
\end{proof}

We begin the study of optimality conditions with a basic result. If $\bar{u}\in \mathbb{U}_{ad}$ denotes a locally optimal control for problem \eqref{def:weak_ocp}--\eqref{eq:weak_st_eq}, then \cite[Lemma 4.18]{Troltzsch}
\begin{equation}
\label{eq:variational_inequality}
j'(\bar u) (u - \bar u) \geq 0 \quad \forall u \in \mathbb{U}_{ad}.
\end{equation}
We recall that, for $u \in \mathbb{U}_{ad}$, $j(u) = J (\mathcal{S}u,u)$. In \eqref{eq:variational_inequality}, $j'(\bar{u})$ denotes the Gate\^aux derivative of  $j$ at $\bar{u}$. To explore \eqref{eq:variational_inequality}, we introduce the \emph{adjoint variable} $p \in W_{0}^{1,r}(\Omega)$, with $r\in (1,d/(d-1))$, as the unique solution to the \emph{adjoint equation}
\begin{equation}\label{eq:adj_eq}
(\nabla w,\nabla p)_{L^2(\Omega)} + \left(\tfrac{\partial a}{\partial y}(\cdot,y)p,w\right)_{L^2(\Omega)} = \sum_{t\in\mathcal{D}}\langle (y(t) - y_{t})\delta_{t},w\rangle\quad \forall w \in W_{0}^{1,r'}(\Omega).
\end{equation}
Here, $\langle \cdot, \cdot \rangle$ denotes the duality pairing between $W^{-1,r}(\Omega):= W^{1,r'}_0(\Omega)'$ and $W^{1,r'}_0(\Omega)$, $\delta_{t}$ corresponds to the Dirac delta supported at the interior point $t \in \Omega$, and $y = \mathcal{S}u$ corresponds to the solution to \eqref{eq:weak_st_eq}; observe that $r'>d$. We notice that, in view of assumptions \ref{A1}--\ref{A3}, problem \eqref{eq:adj_eq} is well--posed; see \cite[Theorem 1]{MR812624}.

We are now in position to present first order necessary optimality conditions.

\begin{theorem}[first order necessary optimality conditions]
\label{thm:optimality_cond}
Assume that \textnormal{\ref{A1}}, \textnormal{\ref{A2}}, and \textnormal{\ref{A3}} hold. Then every locally optimal control $\bar{u} \in \mathbb{U}_{ad}$ for problem \eqref{def:weak_ocp}--\eqref{eq:weak_st_eq} satisfies the variational inequality
\begin{equation}\label{eq:var_ineq}
(\bar{p}+\alpha \bar{u},u-\bar{u})_{L^2(\Omega)}\geq 0 \quad \forall u\in \mathbb{U}_{ad},
\end{equation}
where $\bar p \in W_{0}^{1,r}(\Omega)$, with $r < d/(d-1)$, solves \eqref{eq:adj_eq} with $y$ replaced by $\bar{y} = \mathcal{S} \bar u$.
\end{theorem}
\begin{proof}
A simple computation reveals that \eqref{eq:variational_inequality} can be rewritten as follows:
\begin{equation}\label{eq:derivative_cost}
\sum_{t\in\mathcal{D}}\left(\mathcal{S}\bar{u}(t)-y_t\right)\cdot\mathcal{S}'(\bar{u})(u-\bar{u})(t)+\alpha(\bar{u},u-\bar{u})_{L^2(\Omega)}\geq 0 \qquad \forall u\in\mathbb{U}_{ad}.
\end{equation}

Let us concentrate on the first term of the left hand side of \eqref{eq:derivative_cost}.  Define $z:=\mathcal{S}'(\bar{u})(u-\bar{u})$. Since $u, \bar u \in L^2(\Omega)$, the results of Theorem \ref{thm:properties_C_to_S} guarantees that
$z\in H^2(\Omega) \cap H_0^1(\Omega) \hookrightarrow W_{0}^{1,\mathfrak{q}}(\Omega)$ for every $\mathfrak{q} < 2d/(d-2)$ \cite[Theorem 4.12]{MR2424078}. In particular, since $2d/(d-2) > d$, we have that $z \in W_0^{1,\mathfrak{q}}(\Omega)$ for every $\mathfrak{q} \in (d,2d/(d-2))$. We are thus able to set $w=z$ as a test function in the adjoint problem \eqref{eq:adj_eq} to obtain
\begin{equation}\label{eq:adj_eq_w=z}
(\nabla z,\nabla \bar{p})_{L^2(\Omega)} + \left(\tfrac{\partial a}{\partial y}(\cdot,\bar{y})\bar{p},z\right)_{L^2(\Omega)} = \sum_{t\in\mathcal{D}} (\bar{y}(t) - y_{t})\cdot z(t).
\end{equation}

On the other hand, we would like to set $w = \bar p$ in the problem that  $z=\mathcal{S}'(\bar{u})(u-\bar{u})$ solves. If that were possible, we would obtain
\begin{equation}\label{eq:aux_adjoint_2}
(\nabla z,\nabla \bar{p})_{L^2(\Omega)}+\left(\tfrac{\partial a}{\partial y}(\cdot,\bar{y})z,\bar{p}\right)_{L^2(\Omega)}=(u-\bar{u},\bar{p})_{L^2(\Omega)}.
\end{equation}
However, since $\bar p \in W_0^{1,r}(\Omega)$ with $r < d/(d-1)$, we have that $\bar p  \not\in  H_0^1(\Omega)$ so that \eqref{eq:aux_adjoint_2} must be justified by different means. Let $\{p_n\}_{n\in\mathbb{N}} \subset C_0^{\infty}(\Omega)$ be such that $p_n \rightarrow \bar{p}$ in $W_{0}^{1,r}(\Omega)$ for every $r < d/(d-1)$. Setting $w = p_n$, with $n \in\mathbb{N}$, in the problem that $z=\mathcal{S}'(\bar{u})(u-\bar{u})$ solves yields
\begin{equation}
(\nabla z,\nabla p_n)_{L^2(\Omega)}+\left(\tfrac{\partial a}{\partial y}(\cdot,\bar{y})z,p_n\right)_{L^2(\Omega)}=(u-\bar{u},p_n)_{L^2(\Omega)}.
\label{eq:aux_first_order}
\end{equation}
Since
$
|(u-\bar{u},\bar{p})_{L^2(\Omega)}-(u-\bar{u},p_n)_{L^2(\Omega)}|
\leq \|u-\bar{u}\|_{L^\infty(\Omega)}\| \bar{p}-p_n \|_{L^1(\Omega)} \rightarrow 0
$
as $n \uparrow \infty$, the right hand side of \eqref{eq:aux_first_order} converges to $(u-\bar{u}, \bar p)_{L^2(\Omega)}$. The existence of $\mathsf{m}>0$ such that $| \bar y (x)| \leq \mathsf{m}$ for a.e.~$x \in \Omega$, $z \in H^2(\Omega)$, and \ref{A3} reveal that, as $n \uparrow \infty$,
\begin{equation*}
\left|\left(\tfrac{\partial a}{\partial y}(\cdot,\bar{y})z,\bar{p}\right)_{L^2(\Omega)} - \left(\tfrac{\partial a}{\partial y}(\cdot,\bar{y})z,p_n\right)_{L^2(\Omega)}\right|\leq 
C_{\mathsf{m}} \| z \|_{L^{\infty}(\Omega)} \| \bar{p}-p_n \|_{L^1(\Omega)} \rightarrow 0.
\end{equation*}
Finally,
$
|(\nabla z, \nabla (\bar{p}-p_n))_{L^2(\Omega)}| \leq \| \nabla z \|_{L^{r'}(\Omega)} \|  \nabla (\bar{p}-p_n) \|_{L^{r}(\Omega)} \rightarrow 0
$
as $n \uparrow \infty$ ($r < \frac{d}{d-1}$).

The desired variational inequality \eqref{eq:var_ineq} follows from \eqref{eq:derivative_cost}, \eqref{eq:adj_eq_w=z}, and \eqref{eq:aux_adjoint_2}.
\end{proof}

We present the following projection formula. If $\bar u$ denotes a locally optimal control for problem \eqref{def:weak_ocp}--\eqref{eq:weak_st_eq}, then \cite[section 4.6]{Troltzsch}
\begin{equation}\label{eq:projection_control} 
\bar{u}(x):=\Pi_{[\texttt{a},\texttt{b}]}(-\alpha^{-1}\bar{p}(x)) \textrm{ a.e.}~x \in \Omega,
\end{equation}
where $\Pi_{[\texttt{a},\texttt{b}]} : L^1(\Omega) \rightarrow  \mathbb{U}_{ad}$ is defined by $\Pi_{[\texttt{a},\texttt{b}]}(v) := \min\{ \texttt{b}, \max\{ v, \texttt{a}\} \}$ a.e.~in  $\Omega$. We immediately notice that, since $\bar{p}\in W^{1,r}(\Omega)$ for every $r < d/(d-1)$, then $\bar u \in W^{1,r}(\Omega)$ \cite[Theorem A.1]{MR1786735}; see also \cite[Theorem 4.1]{MR2674627}.

The next result further \FF{illustrates} regularity properties of $\bar{p}$;  see \cite[Theorem 3.4]{MR2434067}. \EO{To present it, we follow \cite{MR3973329} and introduce}
\begin{equation}\label{eq:set_not_zero}
\FF{\mathcal{E}:=\{t\in \mathcal{D}: \bar{y}(t)\neq y_{t}\}.}
\end{equation}
\EO{Observe that, if $t \in \mathcal{D} \setminus \mathcal{E}$, then the coefficient $\bar{y}(t) - y_{t}$ in \eqref{eq:adj_eq} vanishes.}

\begin{theorem}[local regularity]\label{thm:reg_adj_state_local}
Assume that \textnormal{\ref{A1}}, \textnormal{\ref{A2}}, and \textnormal{\ref{A3}} hold. Let $\bar{p}\in W_0^{1,r}(\Omega)$, with $r<d/(d-1)$, be the solution to \eqref{eq:adj_eq} with $y=\bar{y}=\mathcal{S}\bar{u}$. Then, 
\begin{equation}\label{eq:regularity_adjoint}
\bar{p}\in H^2(\Omega\setminus \cup_{t\in\FF{\mathcal{E}}}\bar{B_t})\cap C^{0,1}(\bar \Omega\setminus \cup_{t\in\FF{\mathcal{E}}} B_t).
\end{equation}
Here, $B_t \subset \Omega$ denotes a suitable open ball centered at $t\in \FF{\mathcal{E}}$ of strictly positive radius.
\end{theorem}
\begin{proof}
\EO{Let us assume that $\mathcal{E}\neq \emptyset$; observe that if $\mathcal{E}=\emptyset$, then $\bar{p}\equiv 0$.} We follow the proof in \cite[Lemma 5.2]{MR3973329}. Let $t\in\FF{\mathcal{E}}$ and let $B_{t}$ and $C_t$ be open balls centered at $t$ such that $B_{t} \Subset C_t$. Let $\gamma:\Omega \to [0,1]$ be a smooth function that satisfies 
$
\gamma \equiv 1
$ 
in 
$\Omega \setminus \cup_{t\in\FF{\mathcal{E}}}\bar{C_t}$
and 
$
\gamma \equiv 0
$ 
on
$
\cup_{t\in\FF{\mathcal{E}}}B_t.
$
Define $\Upsilon := \gamma\bar{p}$. Basic computations, on the basis of the fact that $\bar p$ solves  \eqref{eq:adj_eq}, reveal that
\begin{equation}
(\nabla \Upsilon, \nabla w)_{L^2(\Omega)} = -\left(\tfrac{\partial a}{\partial y}(\cdot,\bar{y})\bar{p} \gamma,w\right)_{L^2(\Omega)} - 2(\nabla\bar{p}\cdot\nabla \gamma, w)_{L^2(\Omega)} -(\bar{p}\Delta \gamma,w)_{L^2(\Omega)}
\label{eq:Upsilon}
\end{equation}
for all $w\in W_0^{1,r'}(\Omega)$. Denote by $\mathfrak{f}$ the forcing term of problem \eqref{eq:Upsilon}. Since $\gamma$ is smooth, $a$ satisfies \ref{A3}, and $\bar{p}\in W_0^{1,r}(\Omega)$ for every $r<d/(d-1)$, we conclude that $\mathfrak{f} \in L^r(\Omega)$. The convexity of $\Omega$ thus guarantees that $\Upsilon\in W^{2,r}(\Omega)$ for $r<d/(d-1)$. This, the smoothness of $\gamma$, and the relation $\Upsilon = \gamma \bar p$, allow us to obtain that $\bar{p}\in W^{2,r}(\Omega\setminus \cup_{t\in\FF{\mathcal{E}}}\bar{B_t}) \hookrightarrow H^1(\Omega\setminus \cup_{t\in\FF{\mathcal{E}}}\bar{B_t})$. This regularity property and the underlying construction of $\gamma$ guarantee that $\mathfrak{f}$ now belongs to $L^2(\Omega)$. Consequently, the convexity of $\Omega$ yields $\bar{p}\in H^2(\Omega\setminus \cup_{t\in\FF{\mathcal{E}}} \bar{B_t})$. This, on the basis of a Sobolev embedding reveals that $\bar p \in W^{1,6}(\Omega\setminus \cup_{t\in\FF{\mathcal{E}}} \bar{B_t})$. As a result, there exists $s>d$ such that $\mathfrak{f} \in L^s(\Omega)$. Invoke \cite[Lemma 4.1]{MR3973329} to arrive at $\bar p \in C^{0,1}(\bar \Omega\setminus \cup_{t\in\FF{\mathcal{E}}} B_t )$. This concludes the proof.
\end{proof}

We now present a regularity results for a locally optimal control $\bar u$.

\begin{theorem}[extra regularity of $\bar{u}$]\label{thm:extra_regul_control}
Suppose that assumptions \textnormal{\ref{A1}}, \textnormal{\ref{A2}}, and \textnormal{\ref{A3}} hold. Then, every locally optimal control $\bar{u}$ belongs to $H^1(\Omega) \cap C^{0,1}(\bar{\Omega})$. 
\end{theorem}
\begin{proof}
The proof of the fact that $\bar u \in H^1(\Omega)$ follows directly from the projection formula \eqref{eq:projection_control} and \cite[Lemma 3.3]{MR3264224}. The $C^{0,1}(\bar \Omega)$-regularity of $\bar{u}$ can be found in \cite[Theorem 3.4]{MR2434067} and \cite[Theorem 4.2]{MR2674627}. For the sake of completeness, we briefly present a proof. \FF{If $\mathcal{E}=\emptyset$, the desired regularity follows. Let us thus assume that $\mathcal{E}\neq\emptyset$ and} let $t\in\FF{\mathcal{E}}$. The following asymptotic behavior holds as $\Omega\ni x\to t$:
\[
\FF{|\bar{p}(x)|} \approx C_{1} \log|x-t| +C_2 \text{ if } d=2,
\quad 
\FF{|\bar{p}(x)|} \approx C_{1}|x-t|^{-1}+C_2 \text{ if } d=3,
\quad
C_1 >0;
\]
see \cite[eq. (4.9)]{MR431753} and \cite[Theorem 3.3]{MR0223740}. Here, $C_2 \in \mathbb{R}$. As a result, for every $\mathsf{M}>0$, there exists a ball $B_{t}$ of strictly positive radius centered at $t$ such that $|\bar{p}(x)| > \mathsf{M}$ for $x \in B_{t}$. Since this argument holds for each $t\in\FF{\mathcal{E}}$, we conclude that, for every $\mathsf{M}>0$, there exist $\{ B_t \}_{t \in \FF{\mathcal{E}}}$ such that
$
|\bar{p}(x)| > \mathsf{M}
$
for $x \in \cup_{t\in\FF{\mathcal{E}}}B_{t}$. Set $\mathsf{M}$ sufficiently large and invoke the  projection formula \eqref{eq:projection_control} to deduce that, for $x\in \cup_{t\in\FF{\mathcal{E}}}B_{t}$, we have either $\bar{u}(x)=\texttt{a}$ or $\bar{u}(x)=\texttt{b}$. The $C^{0,1}(\bar \Omega)$-regularity of $\bar{u}$ thus follow from \eqref{eq:regularity_adjoint}.
\end{proof}


\subsection{\EO{Second order optimality conditions}}\label{sec:2nd_order}
In this section, we formulate necessary and sufficient second order optimality conditions. 

We begin our studies with the following result.

\begin{theorem}[$j$ is of class $C^2$ and $j''$ is locally Lipschitz]
\label{thm:diff_properties_j}
Assume that \textnormal{\ref{A1}}, \textnormal{\ref{A2}}, and \textnormal{\ref{A3}} hold. Then, the reduced cost functional $j: L^2(\Omega) \rightarrow \mathbb{R}$ is of class $C^2$. Moreover, for every $u, v_1,v_2 \in L^2(\Omega)$, we have
\begin{equation}\label{eq:charac_j2}
j''(u)(v_1,v_2)=\alpha(v_1,v_2)_{L^2(\Omega)}-\left(\tfrac{\partial^2 a}{\partial y^2}(\cdot,y)z_{v_1}z_{v_2},p\right)_{L^2(\Omega)}+\sum_{t\in \mathcal{D}}z_{v_1}(t)z_{v_2}(t),
\end{equation}
where $p$ solves \eqref{eq:adj_eq} and $z_{v_i}=\mathcal{S}'(u)v_i$, with $i \in \{1,2 \}$. In addition, we have
\begin{equation}\label{eq:continuity_of_j2}
|j''(u_1)v^2-j''(u_2)v^2|
\lesssim
\|u_1-u_2\|_{L^2(\Omega)}\|v\|_{L^2(\Omega)}^2.
\end{equation}
\end{theorem}
\begin{proof}
The fact that $j$ is of class $C^2$ is an immediate consequence of the differentiability properties of the control to state map $\mathcal{S}$ given in Theorem \ref{thm:properties_C_to_S}. It thus suffices to derive \eqref{eq:charac_j2} and \eqref{eq:continuity_of_j2}. To accomplish this task, we begin with a basic computation, which reveals that, for every $u,v_1,v_2\in L^2(\Omega)$, we have 
\begin{equation}\label{eq:charac_j2_prev}
j''(u)(v_1,v_2)
=
\alpha(v_1,v_2)_{L^2(\Omega)}
+
\sum_{t\in \mathcal{D}}
\left[(\mathfrak{z}(t)\cdot(\mathcal{S}u(t)-y_t)+z_{v_1}(t)z_{v_2}(t) \right],
\end{equation}
where $\mathfrak{z},z_{v_1},z_{v_2} \in H^{2}(\Omega)\cap H_0^1(\Omega)$ are as in the statement of Theorem \ref{thm:properties_C_to_S}. Set $w=  \mathfrak{z}$ in \eqref{eq:adj_eq}  and invoke a similar approximation argument to that used in the proof of Theorem \ref{thm:optimality_cond}, that essentially allows us to set $w = p$ in \eqref{eq:aux_adjoint_diff_2}, to obtain
\[
\sum_{t\in \mathcal{D}}\mathfrak{z}(t)\cdot(\mathcal{S}u(t)-y_t)=-\left(\tfrac{\partial^2 a}{\partial y^2}(\cdot,y)z_{v_1}z_{v_2},p\right)_{L^2(\Omega)}.
\]
Replacing the previous identity into \eqref{eq:charac_j2_prev} yields \eqref{eq:charac_j2}.

Let $u_1,u_2,v\in L^2(\Omega)$. Define $\chi = \mathcal{S}'(u_1)v$ and $\psi = \mathcal{S}'(u_2)v$; $\chi$ and $\psi$ correspond to the solutions to \eqref{eq:aux_adjoint} with $y = y_{u_{1}} := \mathcal{S}u_{1}$ and  $y = y_{u_{2}} := \mathcal{S}u_{2}$, respectively. In view of 
the identity \eqref{eq:charac_j2} we obtain
\begin{equation}
\label{eq:identity_ju1_ju2}
\begin{aligned}
j''(u_1)v^2-j''(u_2)v^2
& =
\left(\tfrac{\partial^2 a}{\partial y^2}(\cdot,y_{u_{2}})\psi^2,p_{u_{2}}\right)_{L^2(\Omega)}-\left(\tfrac{\partial^2 a}{\partial y^2}(\cdot,y_{u_{1}})\chi^2,p_{u_{1}}\right)_{L^2(\Omega)}
\\
& + \sum_{t\in \mathcal{D}}( \chi^2(t)-\psi^2(t))
=: \mathbf{I}+\sum_{t\in\mathcal{D}}\mathbf{II}_{t}.
\end{aligned}
\end{equation}
Here, $i \in \{ 1,2\}$ and $p_{u_{i}} \in W_0^{1,r}(\Omega)$, with $r < d/(d-1)$, denotes the solution to  \eqref{eq:adj_eq} with $y$ replaced by $y_{u_{i}}$.  To estimate $\mathbf{I}$, we first rewrite it as follows:
\begin{multline*}
\mathbf{I}=\left(\left[\tfrac{\partial^2 a}{\partial y^2}(\cdot,y_{u_{2}})-\tfrac{\partial^2 a}{\partial y^2}(\cdot,y_{u_{1}})\right]\psi^2,p_{u_{2}}\right)_{L^2(\Omega)}+\left(\tfrac{\partial^2 a}{\partial y^2}(\cdot,y_{u_{1}})\psi^2,p_{u_{2}}-p_{u_{1}}\right)_{L^2(\Omega)}\\
+ \left(\tfrac{\partial^2 a}{\partial y^2}(\cdot,y_{u_{1}})[\psi^2-\chi^2],p_{u_{1}}\right)_{L^2(\Omega)}=:\mathbf{I}_{1}+\mathbf{I}_{2}+ \mathbf{I}_{3}.
\end{multline*}
Invoke \ref{A3}, a generalized H\"{o}lder inequality, the Sobolev embedding $H_0^1(\Omega)\hookrightarrow L^4(\Omega)$, the well-posedness of problem \eqref{eq:aux_adjoint}, and the Lipschitz property \eqref{eq:lipschitz_prop}, to obtain
\begin{equation}\label{eq:estimate_I_a}
\mathbf{I}_{1}
\lesssim
\|y_{u_{1}}-y_{u_{2}}\|_{L^\infty(\Omega)}\|\nabla\psi\|_{L^2(\Omega)}^2\|p_{u_{2}}\|_{L^2(\Omega)}
\lesssim 
\|u_{1}-u_{2}\|_{L^2(\Omega)}\|v\|_{L^2(\Omega)}^2,
\end{equation}
where we have also used the stability estimate
\begin{equation}\label{eq:stab_state_eq}
\|p_{u_{2}}\|_{L^2(\Omega)}
\lesssim
\|\nabla p_{u_{2}}\|_{L^{r}(\Omega)}
\lesssim
\|y_{u_{2}}\|_{L^\infty(\Omega)}+\sum_{t\in\mathcal{D}}|y_{t}|
\lesssim C +\sum_{t\in\mathcal{D}}|y_{t}|.
\end{equation}
Theorem \ref{thm:well_posedness_semilinear} and the assumption on $u_2$ yield $\|y_{u_{2}}\|_{L^\infty(\Omega)} \lesssim \| u_2 -a(\cdot,0) \|_{L^2(\Omega)} \leq C$. To guarantee that $p_{u_{2}} \in L^2(\Omega)$, and thus the first estimate in \eqref{eq:stab_state_eq}, we further restrict the exponent $r$ to belong to $[2d/(d+2),d/(d-1))$ \cite[Theorem 4.12]{MR2424078}. To control $\mathbf{I}_2$, we invoke similar arguments to the ones that lead to \eqref{eq:estimate_I_a}. In fact, we have
\begin{multline*}
\mathbf{I}_{2}
\leq
C_{\mathsf{m}} \| \psi \|^2_{L^4(\Omega)} \|p_{u_{1}}-p_{u_{2}}\|_{L^{2}(\Omega)}
\lesssim 
\|\nabla\psi\|_{L^{2}(\Omega)}^2
\|\nabla(p_{u_{1}}-p_{u_{2}})\|_{L^{r}(\Omega)}
\\
\lesssim 
\|v\|_{L^2(\Omega)}^2\|y_{u_1}-y_{u_2}\|_{L^\infty(\Omega)}
\lesssim
\|v\|_{L^2(\Omega)}^2\|u_{1}-u_{2}\|_{L^2(\Omega)}.
\end{multline*}
Finally, to estimate $\mathbf{I}_3$, we notice that $\psi-\chi \in H_0^1(\Omega) \cap L^{\infty}(\Omega)$ solves
\begin{equation*}
(\nabla (\psi-\chi), \nabla w)+ \left(\tfrac{\partial a}{\partial y}(\cdot,y_{u_2})(\psi-\chi),w\right)_{L^2(\Omega)}=\left(\left[\tfrac{\partial a}{\partial y}(\cdot,y_{u_1})-\tfrac{\partial a}{\partial y}(\cdot,y_{u_2})\right]\chi,w\right)_{L^2(\Omega)}
\end{equation*}
for all $w\in H_0^1(\Omega)$. The estimate
$
\|\psi-\chi\|_{L^\infty(\Omega)} \lesssim \| [ \tfrac{\partial a}{\partial y}(\cdot,y_{u_1})-\tfrac{\partial a}{\partial y}(\cdot,y_{u_2})]\chi  \|_{L^2(\Omega)},
$
combined with \ref{A3} and the Lipschitz property \eqref{eq:lipschitz_prop}, allows us to conclude that
\begin{equation}\label{eq:estimate_tilde_hat_z}
\|\psi-\chi\|_{L^\infty(\Omega)}
\lesssim
\|y_{u_1}-y_{u_2}\|_{L^\infty(\Omega)}\|\chi\|_{L^2(\Omega)}
\lesssim
\|u_{1}-u_{2}\|_{L^2(\Omega)}\|v\|_{L^2(\Omega)}.
\end{equation}
Utilizing \ref{A3}, again, the well-posedness of problem \eqref{eq:aux_adjoint}, and \eqref{eq:estimate_tilde_hat_z} we obtain
$
\mathbf{I}_{3}
\lesssim
\|p_{u_{1}}\|_{L^2(\Omega)}
\|\psi-\chi\|_{L^\infty(\Omega)}(\|\psi\|_{L^2(\Omega)}+\|\chi\|_{L^2(\Omega)})
\lesssim
\|u_{1}-u_{2}\|_{L^2(\Omega)}\|v\|_{L^2(\Omega)}^2,
$
where we have also used the stability bound $\|\psi\|_{L^2(\Omega)}+\|\chi\|_{L^2(\Omega)} \lesssim \|v\|_{L^2(\Omega)}$ and an estimate for $\|p_{u_{1}}\|_{L^2(\Omega)}$ which is similar to the one derived in \eqref{eq:stab_state_eq}.

A collection of the previous bounds yield
$
\mathbf{I}
=
\mathbf{I}_{1}+\mathbf{I}_{2}+\mathbf{I}_{3}
\lesssim
\|u_{1}-u_{2}\|_{L^2(\Omega)}\|v\|_{L^2(\Omega)}^2.
$

Let $t\in\mathcal{D}$. We now estimate $\mathbf{II}_{t}$ in \eqref{eq:identity_ju1_ju2}. Combining the estimate \eqref{eq:estimate_tilde_hat_z} with a stability estimate for \eqref{eq:aux_adjoint}, it immediately follows that
$
\mathbf{II}_{t}
\lesssim
\|\psi-\chi\|_{L^\infty(\Omega)}(\|\psi\|_{L^\infty(\Omega)}+\|\chi\|_{L^\infty(\Omega)})
\lesssim
\|u_{1}-u_{2}\|_{L^2(\Omega)}\|v\|_{L^2(\Omega)}^2.
$

We conclude \eqref{eq:continuity_of_j2} by replacing the obtained estimates for $\mathbf{I}$ and $\mathbf{II}_{t}$ into \eqref{eq:identity_ju1_ju2}.
\end{proof}

Let $\bar{u} \in \mathbb{U}_{ad}$ satisfy the first order optimality conditions \eqref{eq:weak_st_eq}, \eqref{eq:adj_eq}, and \eqref{eq:var_ineq}. Define $\bar{\mathfrak{p}} :=  \bar p + \alpha \bar u$. The variational inequality \eqref{eq:var_ineq} immediately yields, {a.e.~$x\in \Omega$,
\begin{equation}\label{eq:derivative_j}
\bar{\mathfrak{p}}(x) 
= 0  \text{ if } \texttt{a}< \bar{u} < \texttt{b}, 
\qquad
\bar{\mathfrak{p}}(x) \geq  0  \text{ if }\bar{u}=\texttt{a}, 
\qquad
\bar{\mathfrak{p}}(x) \leq  0 \text{ if } \bar{u}=\texttt{b}.
\end{equation}

To formulate second order conditions, we introduce the \emph{cone of critical directions}
\begin{equation}\label{def:critical_cone}
C_{\bar{u}}:=\{v\in L^2(\Omega) \text{ satisfying } \eqref{eq:sign_cond} \text{ and } v(x) = 0 \text{ if } \bar{\mathfrak{p}}(x) \neq 0\},
\end{equation}
where condition \eqref{eq:sign_cond} reads as follows:
\begin{equation}
\label{eq:sign_cond}
v(x)
\geq 0 \text{ a.e.}~x\in\Omega \text{ if } \bar{u}(x)=\texttt{a},
\qquad
v(x)
\leq 0 \text{ a.e.}~x\in\Omega \text{ if } \bar{u}(x)=\texttt{b}.
\end{equation}
From now on, we will restrict $r$ to belong to $[2d/(d+2),d/(d-1))$ so that $p$, the solution to \eqref{eq:adj_eq}, belongs to $L^2(\Omega)$ \cite[Theorem 4.12]{MR2424078}. This implies that $\bar{\mathfrak{p}}\in L^2(\Omega)$.

We are now in position to present second order necessary and sufficient optimality conditions. While it is fair to say that for distributed semilinear optimal control problems such a theory is well-understood, our main source of difficulty here is that, \EO{under the \emph{most likely scenario} where $\mathcal{E} \neq \emptyset $,} $\bar p \notin H_0^1(\Omega) \cap C(\bar \Omega)$: $\bar p \in W_0^{1,r}(\Omega) \setminus H_0^1(\Omega) \cap C(\bar \Omega)$ with $r\in [2d/(d+2),d/(d-1))$.

\begin{theorem}[second order necessary optimality condition]
\label{thm:nec_opt_cond}
If $\bar{u}\in \mathbb{U}_{ad}$ denotes a locally optimal control for problem \eqref{def:weak_ocp}--\eqref{eq:weak_st_eq}, then
$
j''(\bar{u})v^2 \geq 0 
$
for all $v\in C_{\bar{u}}$.
\end{theorem}
\begin{proof}
The proof essentially follows the same arguments as those elaborated in the proof of \cite[Theorem 23]{MR3586845}. For brevity, we skip details.
\end{proof}

We now derive, in the spirit of \cite[Theorem 23]{MR3586845}, a sufficient optimality condition with a minimal gap with respect to the necessary one stated in Theorem \ref{thm:nec_opt_cond}.

\begin{theorem}[second order sufficient optimality condition]\label{thm:optimal_solution}
Let $(\bar{y},\bar{p},\bar{u})$ be a triplet satisfying the first order optimality conditions \eqref{eq:weak_st_eq}, \eqref{eq:adj_eq}, and \eqref{eq:var_ineq}. If $j''(\bar{u})v^2 > 0$ for all $ v\in C_{\bar{u}}\setminus \{0\}$, then there exist $\mu > 0$ and $\sigma > 0$ such that
\begin{equation}\label{eq:optimal_minimum}
j(u)\geq j(\bar{u})+\tfrac{\mu}{2}\|u-\bar{u}\|_{L^2(\Omega)}^2\quad \forall u\in \mathbb{U}_{ad}: \|u-\bar{u}\|_{L^2(\Omega)}\leq \sigma.
\end{equation}
In particular, $\bar{u}$ is a locally optimal control in the sense of $L^2(\Omega)$.
\end{theorem}
\begin{proof}
We follow \cite[Theorem 23]{MR3586845} and proceed by contradiction. Assume that \eqref{eq:optimal_minimum}  does not hold. Hence, for any $k\in\mathbb{N}$, there is an element $u_{k}\in \mathbb{U}_{ad}$ such that
\begin{equation}\label{eq:contradic_I}
\|\bar{u}-u_{k}\|_{L^2(\Omega)}< \tfrac{1}{k},
\qquad 
j(u_{k}) < j(\bar{u}) + \tfrac{1}{2k}\|\bar{u}-u_{k}\|_{L^2(\Omega)}^2.
\end{equation}
Define 
$
\rho_{k}:=\|u_{k}-\bar{u}\|_{L^2(\Omega)}
$
and
$
v_{k}:=(u_{k}-\bar{u})/\rho_{k}.
$
We assume that (up to a subsequence if necessary) $v_{k}\rightharpoonup v$ in $L^2(\Omega)$. In what follows we prove that $v\in C_{\bar{u}}$ and that $v=0$.

Since the set of elements satisfying condition \eqref{eq:sign_cond} is closed and convex in $L^2(\Omega)$, it is weakly closed. Consequently, $v$ satisfies \eqref{eq:sign_cond} as well. To verify the remaining condition in \eqref{def:critical_cone}, we invoke the mean value theorem and \eqref{eq:contradic_I} to arrive at 
\begin{equation}\label{eq:contradic_II}
 j'(\tilde{u}_{k})v_{k}  =\tfrac{1}{\rho_k}(j(u_{k}) - j(\bar{u}))< \tfrac{\rho_{k}}{2k} \to 0,
\quad
k \uparrow \infty,
\end{equation}
where $\tilde{u}_{k}=\bar{u}+ \theta_{k}(u_{k} - \bar{u})$ and $\theta_{k} \in (0,1)$. Define $\tilde{y}_{k}:= \mathcal{S}\tilde{u}_{k}$ and $\tilde{p}_{k}$ as the unique solution to \eqref{eq:adj_eq} with $y=\tilde{y}_{k}$. Since $\tilde u_k \rightarrow \bar u$ in $L^2(\Omega)$ as $k \uparrow \infty$, an application of Theorem \ref{thm:lipschitz_property} yields $\tilde{y}_{k} \to \bar{y}$ in $H_0^1(\Omega) \cap C(\bar \Omega)$ as $k \uparrow \infty$. This, in view of \cite[Theorem 1]{MR812624}, implies that $\tilde{p}_{k}\to \bar{p}$ in $W_0^{1,r}(\Omega)$, for every $r < d/(d-1)$, as $k \uparrow\infty$. In particular, we have $\tilde{p}_{k}\to \bar{p}$ in $L^2(\Omega)$ as $k \uparrow\infty$. Consequently, since $\tilde{\mathfrak{p}}_k:= \tilde p_k + \alpha \tilde u_k \rightarrow \bar{\mathfrak{p}} = \bar p + \alpha \bar u$ in $L^2(\Omega)$ and $v_{k}\rightharpoonup v$ in $L^2(\Omega)$, as $k \uparrow\infty$, we invoke \eqref{eq:contradic_II} to obtain
\begin{equation*}
j'(\bar{u})v  = \int_{\Omega} \bar{\mathfrak{p}}(x) v(x) \mathrm{d}x = \lim_{k \uparrow \infty}
\int_{\Omega} \tilde{\mathfrak{p}}_k(x) v_k(x) \mathrm{d}x = \lim_{k \uparrow \infty}j'(\tilde{u}_{k})v_{k} \leq 0.
\end{equation*}
On the other hand, in view of \eqref{eq:var_ineq} we obtain $\int_{\Omega} \bar{\mathfrak{p}}(x) v_k(x) = \rho_k^{-1} \int_{\Omega} \bar{\mathfrak{p}}(x) (u_k(x) - \bar{u}(x)) \mathrm{d}x \geq 0$. This implies $\int_{\Omega} \bar{\mathfrak{p}}(x) v(x) \mathrm{d}x \geq 0$. Consequently, $\int_{\Omega} \bar{\mathfrak{p}}(x) v(x) \mathrm{d}x = 0$. Proceeding as in the proof of \cite[Theorem 23]{MR3586845}, we can thus conclude that $v\in C_{\bar{u}}$.

We now prove that $v=0$. To accomplish this task, we invoke Taylor's theorem, the inequality in \eqref{eq:contradic_I}, and $j'(\bar{u})(u_{k}-\bar{u})\geq 0$, for every $k\in\mathbb{N}$, to arrive at
\begin{equation*}
\tfrac{\rho_{k}^2}{2}j''(\hat{u}_{k})v_{k}^2
=
j(u_k) - j(\bar u)-j'(\bar u)(u_k - \bar u)
\leq 
j({u}_{k})-j(\bar{u}) < \tfrac{\rho_{k}^2}{2k},
\end{equation*}
where, for $k\in\mathbb{N}$, $\hat{u}_{k}=\bar{u}+\theta_{k}(u_{k} - \bar{u})$ with $\theta_{k}\in (0,1)$.
Thus, $j''(\hat{u}_{k})v_{k}^2 < k^{-1} \rightarrow 0$ as $k \uparrow \infty$. We now prove that $j''(\bar u )v^2 \leq \liminf_{k} j''(\hat{u}_{k})v_{k}^2$. Let $\hat{z}_{v_k}$ and $z_v$ be the solutions to \eqref{eq:aux_adjoint} with forcing terms $v_k$ and $v$, respectively. Invoke \eqref{eq:charac_j2} and write
\[
j''(\hat{u}_{k})v_{k}^2
=
\alpha \| v_k \|^2_{L^2(\Omega)}
-
\left(\tfrac{\partial^2 a}{\partial y^2}(\cdot,\hat{y}_k)\hat{z}^2_{v_k},\hat{p}_k\right)_{L^2(\Omega)}
+
\sum_{t\in \mathcal{D}}\hat{z}^2_{v_k}(t).
\]
Since $v_k \rightharpoonup v$ in $L^2(\Omega)$ implies that $\hat{z}_{v_k} \rightharpoonup z_v$ in $H^2(\Omega) \cap H_0^1(\Omega)$ as $k \uparrow \infty$, the compact embedding $H^2(\Omega) \hookrightarrow C(\bar \Omega)$ yields $\sum_{t \in \mathcal{D}}\hat{z}^2_{v_k}(t) \rightarrow \sum_{t \in \mathcal{D}}z^2_{v}(t)$. In addition, we have
\begin{multline}
\left | \int_{\Omega} \left( \tfrac{\partial^2 a}{\partial y^2}(x,\bar y) z^2_{v} \bar p - \tfrac{\partial^2 a}{\partial y^2}(x,\hat{y}_k)\hat{z}^2_{v_k} \hat{p}_k \right) \mathrm{d}x \right|
\leq  C_{\mathsf{m}} \| z_{v} \|^2_{L^{\infty}(\Omega)} \EO{\|  \bar p -  \hat{p}_k\|_{L^1(\Omega)}}
\\
+ C_{\mathsf{m}} \| \hat{p}_k\|_{L^1(\Omega)} \left( 
\|  \bar y -  \hat{y}_k\|_{L^{\infty}(\Omega)} \| z_{v} \|^2_{L^{\infty}(\Omega)}
+
 \| z_{v} + \hat{z}_{v_k}\|_{L^{\infty}(\Omega)}  \| z_{v} - \hat{z}_{v_k}\|_{L^{\infty}(\Omega)}
 \right) \rightarrow 0
\label{eq:jpp_aux_1}
\end{multline}
as $k \uparrow \infty$. To obtain \eqref{eq:jpp_aux_1}, we used \ref{A3}, $\tilde{p}_{k}\to \bar{p}$ in $W_0^{1,r}(\Omega)$, for every $r < d/(d-1)$, and  $\tilde{y}_{k} \to \bar{y}$ in $H_0^1(\Omega) \cap C(\bar \Omega)$ as $k \uparrow \infty$. Since the square of $\| v \|_{L^2(\Omega)}$ is weakly lower semicontinuous in $L^2(\Omega)$, we can thus deduce that $j''(\bar u )v^2 \leq \liminf_{k} j''(\hat{u}_{k})v_{k}^2$. Therefore, since $j''(\bar u )v^2 \leq \liminf_{k}j''(\hat{u}_{k})v_{k}^2 \leq 0$ and $v \in C_{\bar u}$, the second order optimality condition $j''(\bar u)v^2 >0$ for all $v \in C_{\bar u} \setminus \{ 0 \}$ immediately yields $v=0$.

Finally, since $v = 0$, we have $\hat{z}_{v_k} \rightharpoonup 0$ in $H^2(\Omega) \cap H_0^1(\Omega)$ as $k \uparrow \infty$. In view of
\[
\alpha = \alpha \| v_k \|^2_{L^2(\Omega)} = j''(\hat{u}_{k})v_{k}^2 
+
\left(\tfrac{\partial^2 a}{\partial y^2}(\cdot,\hat{y}_k)\hat{z}^2_{v_k},\hat{p}_k\right)_{L^2(\Omega)}
-
\sum_{t\in \mathcal{D}}\hat{z}^2_{v_k}(t),
\]
$\liminf_{k}j''(\hat{u}_{k})v_{k}^2 \leq 0$ yields $\alpha \leq 0$. This is a contradiction and concludes the proof.
\end{proof}

Let us introduce
$
C_{\bar{u}}^\tau:=\{v\in L^2(\Omega) \textnormal{ satisfying \eqref{eq:sign_cond} and } v(x)=0 \textnormal{ if } |\bar{\mathfrak{p}}(x)|>\tau \}
$
and present the following result.

\begin{theorem}[equivalent optimality conditions]\label{thm:equivalent_opt_cond}
If $(\bar{y},\bar{p},\bar{u})$ denotes a triplet satisfying \eqref{eq:weak_st_eq}, \eqref{eq:adj_eq}, and \eqref{eq:var_ineq}, then the following statements are equivalent:
\begin{equation}\label{eq:second_order_2_2}
j''(\bar{u})v^2 > 0 \quad \forall v \in C_{\bar{u}}\setminus \{0\}
\end{equation}
and
\begin{equation}\label{eq:second_order_equivalent}
\exists \mu, \tau >0: \quad j''(\bar{u})v^2 \geq \mu \|v\|_{L^2(\Omega)}^2 \quad \forall v \in C_{\bar{u}}^\tau.
\end{equation}
\end{theorem}
\begin{proof}
The proof essentially follows the same arguments as those elaborated in the proof of \cite[Theorem 25]{MR3586845}. For brevity, we skip details.
\end{proof}


\section{Finite element approximation}\label{sec:fem}

We now introduce the discrete setting in which we will operate \cite{MR2373954,CiarletBook,Guermond-Ern}. We denote by $\mathscr{T}_h = \{ T\}$ a conforming partition, or mesh, of $\bar{\Omega}$ into closed simplices $T$ with size $h_T = \text{diam}(T)$. Define $h:=\max_{ T \in \mathscr{T}_h} h_T$. We denote by $\mathbb{T} = \{\mathscr{T}_h \}_{h>0}$ a collection of conforming and quasi-uniform meshes $\mathscr{T}_h$.

Given a mesh $\mathscr{T}_{h} \in \mathbb{T}$, we define the finite element space 
\begin{equation}\label{def:piecewise_linear_set}
\mathbb{V}_{h}:=\{v_{h}\in C(\bar{\Omega}): v_{h}|_T\in \mathbb{P}_{1}(T) \ \forall T\in \T_{h}\}\cap H_0^1(\Omega).
\end{equation}

In the following sections we present finite element approximations of the state and adjoint equations and the optimal control problem \eqref{def:weak_ocp}--\eqref{eq:weak_st_eq}.

\subsection{Discrete state equation}\label{sec:fem_state}

Let $f \in L^2(\Omega)$. We define the Galerkin approximation of the solution $\mathsf{y}$ to problem \eqref{eq:weak_semilinear_pde} by
\begin{equation}\label{eq:discrete_semilinear}
\mathsf{y}_h\in\mathbb{V}_h : 
\quad
(\nabla \mathsf{y}_h, \nabla v_{h})_{L^2(\Omega)} + (a(\cdot,\mathsf{y}_h),v_{h})_{L^2(\Omega)} = (f,v_{h})_{L^2(\Omega)} 
\quad \forall v_h \in \mathbb{V}_h.
\end{equation}
Let $a = a(x,y) : \Omega \times \mathbb{R} \rightarrow \mathbb{R}$ be a Carath\'eodory function that is monotone increasing and locally Lipschitz in $y$, a.e.~in $\Omega$, with $a(\cdot,0) \in L^2(\Omega)$. An application of Brouwer's fixed point theorem yields the existence of a unique solution to \eqref{eq:discrete_semilinear}. In addition, we have $\| \nabla \mathsf{y}_h \|_{L^2(\Omega)} \lesssim \| f- a(\cdot,0) \|_{L^2(\Omega)}$; see \cite[Theorem 3.2]{MR3333667} and \cite[section 7]{MR3586845}.

We now provide a local regularity result for the solution $\mathsf{y}$ of problem \eqref{eq:weak_semilinear_pde} that will be of importance to derive error estimates. Let $\Omega_1 \Subset \Omega_0 \Subset \Omega$ with $\Omega_0$ smooth. Let $f \in L^2(\Omega) \cap L^t(\Omega_0)$, where $t \in [2,\infty)$. Since $\mathsf{y}$ can be seen as the solution to
\[
\mathsf{y} \in H_0^1(\Omega): \quad (\nabla \mathsf{y}, \nabla v)_{L^2(\Omega)} = (f - a(\cdot,\mathsf{y}),v)_{L^2(\Omega)} \quad \forall v \in H_0^1(\Omega),
\]
we can invoke \cite[Lemma 4.2]{MR3973329} to deduce that
\begin{equation}
\| \mathsf{y} \|_{W^{2,t}(\Omega_1)} \leq C_t \left( \| f - a(\cdot,\mathsf{y}) \|_{L^t(\Omega_0)}
+
\| f - a(\cdot,\mathsf{y})   \|_{L^{2}(\Omega)} \right),
\label{eq:reg_W2t}
\end{equation}
where $C_t$ behaves as $C t$, with $C>0$, as $t \uparrow \infty$. Notice that we have further assumed that $a$ satisfies $a(\cdot,0) \in L^{t}(\Omega_0)$, which, since $a = a(x,y)$ is locally Lipschitz in $y$, implies that $\| a(\cdot,\mathsf{y}) \|_{L^{t}(\Omega_0)} \lesssim \| f \EO{-a(\cdot,0)} \|_{L^2(\Omega)} + \| a(\cdot,0) \|_{L^{t}(\Omega_0)}$.

\begin{theorem}[a priori error estimates]\label{thm:error_estimates_state}
Let $\Omega \subset \mathbb{R}^d$ be an open, bounded, and convex polytope. Let $a = a(x,y) : \Omega \times \mathbb{R} \rightarrow \mathbb{R}$ be a Carath\'eodory function that is monotone increasing and locally Lipschitz in $y$ with $a(\cdot,0) \in L^2(\Omega)$. Let $\mathsf{y} \in H_0^1(\Omega)\cap H^2(\Omega)$ and $\mathsf{y}_{h}\in\mathbb{V}_h$ be the solutions to \eqref{eq:weak_semilinear_pde} and \eqref{eq:discrete_semilinear}, respectively, with $f \in L^2(\Omega)$. If $h$ is sufficiently small, we thus have the following error estimates:
\begin{equation}\label{eq:global_estimate_state}
\|\mathsf{y}-\mathsf{y}_{h}\|_{L^2(\Omega)}\lesssim h^2 \|f-a(\cdot,0)\|_{L^2(\Omega)}, 
\end{equation}
and
\begin{equation}\label{eq:global_estimate_state_II}
\|\mathsf{y}-\mathsf{y}_{h}\|_{L^\infty(\Omega)} \lesssim h^{2-\frac{d}{2}} \|f-a(\cdot,0)\|_{L^2(\Omega)}.
\end{equation}
Let $\Omega_1 \Subset \Omega_0 \Subset \Omega$ with $\Omega_0$ smooth. If, in addition, $f, a(\cdot,0) \in L^{\infty}(\Omega_0)$, then
\begin{equation}\label{eq:local_estimate_state}
\|\mathsf{y}-\mathsf{y}_{h}\|_{L^\infty(\Omega_1)}\lesssim h^2|\log h|^{2}.
\end{equation}
In all estimates the hidden constant is independent of $h$.
\end{theorem}
\begin{proof}
We refer the reader to \cite[Lemma 4]{MR2009948} and \cite[Theorem 1]{MR2009948} for a proof of the estimates \eqref{eq:global_estimate_state} and \eqref{eq:global_estimate_state_II}, respectively; see also \cite[Theorem 2]{MR2009948}. We provide a proof of \eqref{eq:local_estimate_state} that is inspired in the arguments developed in \cite[Theorem 3.5]{MR3333667} and \cite[Lemma 4.4]{MR3973329}. We begin with a simple application of the triangle inequality and write
\[
\|\mathsf{y}-\mathsf{y}_{h}\|_{L^\infty(\Omega_1)} \leq \|\mathsf{y}- \mathfrak{y}_{h} \|_{L^\infty(\Omega_1)} + \|\mathfrak{y}_{h} -\mathsf{y}_h \|_{L^\infty(\Omega_1)},
\]
where $\mathfrak{y}_h$ solves \eqref{eq:discrete_semilinear} with $a(\cdot,\mathsf{y}_h)$ replaced by $a(\cdot,\mathsf{y})$. Let $\Lambda_1$ be a smooth domain such that $\Omega_1 \Subset \Lambda_1 \Subset \Omega_0$. Since $(\nabla(\mathsf{y}-\mathfrak{y}_{h}),\nabla v_h)_{L^2(\Omega)} = 0$ for all $v_h \in \mathbb{V}_h$, we invoke \cite[Corollary 5.1]{MR431753} to obtain the existence of $h_0 \in (0,1)$ such that
\[
\|\mathsf{y} - \mathfrak{y}_{h} \|_{L^\infty(\Omega_1)} \lesssim   |\log h| \| \mathsf{y} - v_h \|_{L^{\infty}(\Lambda_1)} + {\ell}^{-d/2} \| \mathsf{y} - \mathfrak{y}_{h}  \|_{L^2(\Omega)},
\quad
v_h \in \mathbb{V}_h,
\]
for every  $h \leq h_0$. Here, $\ell$ is such that $\mathrm{dist}(\Omega_1,\partial \Lambda_1) \geq \ell$, $\mathrm{dist}(\Lambda_1,\partial \Omega_{\EO{0}}) \geq \ell$, and $C h \leq \ell$, where $C>0$. Since $f, a(\cdot,0) \in L^{\infty}(\Omega_0) \cap L^2(\Omega)$, the regularity estimate  \eqref{eq:reg_W2t} implies that $\mathsf{y} \in W^{2,t}(\Lambda_1) \cap H^2(\Omega)$ for every $t < \infty$. Thus
\begin{multline*}
\|\mathsf{y}- \mathfrak{y}_{h} \|_{L^\infty(\Omega_1)} \leq C_1 |\log h|  t h^{2-\frac{d}{t}} \left[\| f - a(\cdot,\mathsf{y}) \|_{L^{\infty}(\Omega_0)} + \| f - a(\cdot,\mathsf{y})   \|_{L^{2}(\Omega)} \right]
\\
+ C_2 h^2  \| f - a(\cdot,\mathsf{y})   \|_{L^{2}(\Omega)}, \quad C_1, C_2 >0.
\end{multline*}
Inspired by \cite[page 3]{MR637283}, we thus set $t = |\log h|$ to arrive at the local estimate $\|\mathsf{y}- \mathfrak{y}_{h} \|_{L^\infty(\Omega_1)} \lesssim h^2| \log h|^2( \| f - a(\cdot,\mathsf{y}) \|_{L^{\infty}(\Omega_0)}+ \| f - a(\cdot,\mathsf{y})   \|_{L^{2}(\Omega)} )$. To control $\|\mathfrak{y}_{h} -\mathsf{y}_h \|_{L^\infty(\Omega_1)}$ 
we observe that
\[
\mathfrak{y}_{h} -\mathsf{y}_h \in \mathbb{V}_h:
\quad
(\nabla (\mathfrak{y}_{h} -\mathsf{y}_h), \nabla v_h)_{L^2(\Omega)} = ( a(\cdot,\mathsf{y}_h) - a(\cdot,\mathsf{y}),v_h)_{L^2(\Omega)} 
\quad
\forall v_h \in \mathbb{V}_h.
\]
Define $\mathfrak{y} \in H_0^1(\Omega)$ as the solution to $(\nabla \mathfrak{y}, \nabla v)_{L^2(\Omega)} = ( a(\cdot,\mathsf{y}_h) - a(\cdot,\mathsf{y}),v)_{L^2(\Omega)}$ for all $v \in H_0^1(\Omega)$. Observe that $\mathfrak{y}_{h} -\mathsf{y}_h$ can be seen as the finite element approximation of $\mathfrak{y}$ within $\mathbb{V}_h$. We thus proceed as follows: $\| \mathfrak{y}_{h} - \mathsf{y}_h \|_{L^\infty(\Omega_1)} \leq \| \mathfrak{y} - (\mathfrak{y}_{h} - \mathsf{y}_h) \|_{L^\infty(\Omega_1)} + \|  \mathfrak{y} \|_{L^\infty(\Omega_1)}$. Invoke a stability estimate for the problem that $\mathfrak{y}$ solves, a basic error estimate, and the Lipschitz property of $a=a(x,y)$ in $y$ to obtain
\begin{align*}
\| \mathfrak{y}_{h} - \mathsf{y}_h \|_{L^\infty(\Omega_1)} 
&\lesssim \| a(\cdot,\mathsf{y}_h) - a(\cdot,\mathsf{y})\|_{L^2(\Omega)}
\\
&\lesssim \| \mathsf{y}_h - \mathsf{y} \|_{L^2(\Omega)} \lesssim h^2 \left( \| f \|_{L^2(\Omega)} + \| a(\cdot, \mathsf{y}) \|_{L^2(\Omega)} \right),
\end{align*}
upon using \eqref{eq:global_estimate_state}. \EO{Observe that $\mathsf{y}_h$ is uniformly bounded. In fact, \eqref{eq:global_estimate_state_II} yields
\[
\|\mathsf{y}_h\|_{L^\infty(\Omega)}
\leq 
\|\mathsf{y} - \mathsf{y}_h\|_{L^\infty(\Omega)} + \|\mathsf{y}\|_{L^\infty(\Omega)} 
\lesssim
(h^{2-\frac{d}{2}} +1) \|f - a(\cdot,0)\|_{L^2(\Omega)}.
\]
}This concludes the proof.
\end{proof}


\subsection{Discrete adjoint equation}\label{sec:fem_adjoint}

Let $u\in\mathbb{U}_{ad}$ and $\{y_{t}\}_{t\in\mathcal{D}}\subset \mathbb{R}$. We define the Galerkin approximation to the adjoint equation \eqref{eq:adj_eq} by
\begin{equation}\label{eq:discrete_adjoint_problem}
q_{h}\in\mathbb{V}_{h}: \ (\nabla w_{h},\nabla q_{h})_{L^2(\Omega)}+\left(\tfrac{\partial a}{\partial y}(\cdot,y)q_{h},w_{h}\right)_{L^2(\Omega)} =  \sum_{t\in\mathcal{D}}\langle(y(t)-y_{t})\delta_{t}, w_h\rangle
\end{equation}
for all $w^{}_{h}\in \mathbb{V}_{h}$, where $y\in H_{0}^1(\Omega)\cap L^\infty(\Omega)$ denotes the solution to \eqref{eq:weak_st_eq} with $u\in\mathbb{U}_{ad}$, i.e., $y=\mathcal{S}u$. Standard results yield the existence and uniqueness of $q_{h}\in\mathbb{V}_{h}$.

We present the following error estimates.

\begin{theorem}[error estimates]\label{thm:error_estimate_fund}
Let $a = a(x,y) : \Omega \times \mathbb{R} \rightarrow \mathbb{R}$ be a Carath\'eodory function of class $C^1$ with respect to the second variable such that $a(\cdot,0) \in L^2(\Omega)$. Assume that \textnormal{\ref{A2}} holds and that, for all $\mathsf{m}>0$, $|\partial a/\partial y (x,y)| \leq C_{\mathsf{m}}$ for a.e.~$x \in \Omega$ and $y \in [-\mathsf{m},\mathsf{m}]$. Let $p\in W_{0}^{1,r}(\Omega)$, with $r \in [2d/(d+2),d/(d-1))$, and $q_h\in\mathbb{V}_h$ be the solutions to \eqref{eq:adj_eq} and \eqref{eq:discrete_adjoint_problem}, respectively. Then
\begin{equation}\label{eq:error_estimate_fundamental_II}
\|p-q_h\|_{L^2(\Omega)}
\lesssim h^{2-\frac{d}{2}} 
\sum_{t \in \mathcal{D}} |y(t) - y_t | 
\lesssim
h^{2-\frac{d}{2}}
\left[ \| u - a(\cdot,0) \|_{L^2(\Omega)} + \sum_{t\in\mathcal{D}}|y_{t}| \right]
\end{equation}
and
\begin{equation}\label{eq:error_estimate_fundamental}
\|p-q_h\|_{L^1(\Omega)}\lesssim h^2|\log h|^2.
\end{equation}
In both estimates the hidden constants are independent of $h$.
\end{theorem}
\begin{proof}
Define $\mathfrak{a}(x) = \partial a/\partial y (x,y(x))$, where $y = \mathcal{S}u$ and $u \in \mathbb{U}_{ad}$. Since $\mathfrak{a} \in L^{\infty}(\Omega)$ and $\mathfrak{a}(x) \geq 0$ for a.e.~$x \in\Omega$, we can apply  \cite[Theorem 3]{MR812624} in combination with Theorem \ref{thm:well_posedness_semilinear} to deduce \eqref{eq:error_estimate_fundamental_II}. The proof of the estimate \eqref{eq:error_estimate_fundamental} follows similar arguments as the ones developed in \cite{MR0471370} and \cite[Lemma 5.3]{MR3973329}. Let $\mathfrak{w}$ be the solution to
\[
\mathfrak{B}(\mathfrak{w},v):=
(\nabla \mathfrak{w},\nabla v)_{L^2(\Omega)} + \left(\tfrac{\partial a}{\partial y}(\cdot,y)  \mathfrak{w},v \right)_{L^2(\Omega)} = (\mathfrak{f},v)_{L^2(\Omega)} \quad \forall v \in H_0^1(\Omega).
\]
Let $\mathfrak{w}_h$ be the Ritz projection of $\mathfrak{w}$ within $\mathbb{V}_h$, i.e., $\mathfrak{w}_h \in \mathbb{V}_h$ is such that $\mathfrak{B}(\mathfrak{w}_h,v_h) = \mathfrak{B}(\mathfrak{w},v_h) $ for all $v_h \in \mathbb{V}_h$. Let $\mathfrak{f}= \mathrm{sgn}(p-q_h)$.  Thus,
\[
\|p-q_h\|_{L^1(\Omega)} = \int_{\Omega} \mathfrak{f}(p-q_h) \mathrm{d} x = \mathfrak{B}(\mathfrak{w},p) - \mathfrak{B}(\mathfrak{w}_h,q_h) = \sum_{t \in \mathcal{D}} (y(t) - y_t) (\mathfrak{m}- \mathfrak{m}_h)(t),
\]
where we have used that $p$ and $q_h$ solve \eqref{eq:adj_eq} and \eqref{eq:discrete_adjoint_problem}, respectively. Since $\mathcal{D} \Subset \Omega$,  similar arguments to the ones used in the proof of \eqref{eq:local_estimate_state} yield 
\[
\|p-q_h\|_{L^1(\Omega)} \lesssim h^2 |\log h|^2 \left(\| y \|_{L^{\infty}(\Omega)} + \sum_{t \in \mathcal{D}} |y_t | \right) \|  \mathfrak{f} \|_{L^{\infty}(\Omega)}.
\]
This concludes the proof.
\end{proof}

The next result will be of importance to obtain an error estimate for a solution technique of \eqref{def:weak_ocp}--\eqref{eq:weak_st_eq} based on the variational discretization approach.

\begin{theorem}[auxiliary estimate]\label{thm:error_estimate_adj}
Assume that \textnormal{\ref{A1}}, \textnormal{\ref{A2}}, and \textnormal{\ref{A3}} hold. Let $p\in W_{0}^{1,r}(\Omega)$ and $q_h\in\mathbb{V}_h$ be the solutions to \eqref{eq:adj_eq} and \eqref{eq:discrete_adjoint_problem}, respectively. Then
\begin{equation}\label{eq:error_estimate_fundamental_III}
\|p-q_h\|_{L^2(\Omega\setminus \cup_{t\in\mathcal{D}}\bar{B_{t}})}
\lesssim h^{2}|\log h|^2 \qquad \forall h<h_{*}.
\end{equation}
Here, $B_{t}\subset\Omega$ denotes a suitable open ball centered at $t\in\mathcal{D}$ of strictly positive radius.
\end{theorem}
\begin{proof}
The proof of \eqref{eq:error_estimate_fundamental_III} follows from a combination of the arguments elaborated in the proof of Theorem \ref{thm:error_estimate_fund} and \cite[Lemma 5.5]{MR3973329}. For brevity, we skip details.
\end{proof}

Let $y_{h}\in\mathbb{V}_{h}$ be the unique solution to the discrete problem \eqref{eq:discrete_semilinear} with $f=u_{h}\in \mathbb{U}_{ad} \subset L^\infty(\Omega)$ and define  the discrete variable $p_{h}\in\mathbb{V}_{h}$ as the unique solution to
\begin{equation}
\label{eq:discrete_adjoint_equation}
(\nabla w^{}_{h},\nabla p_{h})_{L^2(\Omega)}\!+\!\left(\tfrac{\partial a}{\partial y}(\cdot,y_{h})p_{h},w^{}_{h}\right)_{L^2(\Omega)}    \!=\!  \sum_{t\in\mathcal{D}}\langle(y_h(t)-y_{t})\delta_{t}, w^{}_{h}\rangle \quad \forall w^{}_{h} \in \mathbb{V}_{h}. \hspace{-0.2cm}
\end{equation}

We present the following error estimate, which will be of importance to perform an a priori error analysis for a suitable discretization of our optimal control problem.

\begin{theorem}[auxiliary error estimate]\label{thm:error_estimates_adjoint_aux}
Let  $a$ be as in the statement of Theorem \ref{thm:error_estimate_fund}. Let $u,{u}_{h}\in L^{2}(\Omega)$ be such that $\|u\|_{L^2(\Omega)} \leq C$ and $\|u_{h}\|_{L^2(\Omega)}  \leq C$, for every $h>0$, where $C>0$.
Let $p$ and ${p}_{h}$ be the solutions to \eqref{eq:adj_eq} and \eqref{eq:discrete_adjoint_equation} with $y=y(u)$ and $y_{h}=y_{h}(u_{h})$, respectively. Then, we have the following error estimate:
\begin{equation}\label{eq:estimate_adj_discrete_continuous}
\|p-p_h\|_{L^2(\Omega)}\lesssim \|u - u_h\|_{L^2(\Omega)} + h^{2-\frac{d}{2}},
\end{equation}
with a hidden constant that is independent of $h$.
\end{theorem}
\begin{proof}
Let $\hat p\in W_0^{1,r}(\Omega)$, with $r\in [2d/(d+2),d/(d-1))$, be such that
\begin{equation}\label{eq:hat_p}
(\nabla w,\nabla \hat{p})_{L^2(\Omega)}+\left(\tfrac{\partial a}{\partial y}(\cdot,y_{h})\hat{p},w\right)_{L^2(\Omega)} \!  = \! \sum_{t\in\mathcal{D}}\langle(y_h(t)-y_{t})\delta_{t}, w\rangle \quad \forall w \in W_0^{1,r'}(\Omega).
\end{equation}
With this variable at hand, we estimate
$
\|p-p_h\|_{L^2(\Omega)} 
\leq
\|p-\hat{p}\|_{L^2(\Omega)} + \|\hat{p}-p_{h}\|_{L^2(\Omega)}.
$

Let us first bound $\|\hat{p}-p_{h}\|_{L^2(\Omega)}$. Since $p_{h}$ corresponds to the finite element approximation of the auxiliary variable $\hat{p}$ within $\mathbb{V}_{h}$, estimate \eqref{eq:error_estimate_fundamental_II} allows us to obtain
\begin{equation}\label{eq:aux_estimate_vi}
\|\hat{p}-p_{h}\|_{L^2(\Omega)} 
\lesssim 
h^{2-\frac{d}{2}}\sum_{t\in\mathcal{D}}|y_{h}(t)-y_{t}| 
\lesssim
h^{2-\frac{d}{2}}\left(\|y_{h}\|_{L^\infty(\Omega)} + \sum_{t\in\mathcal{D}}|y_{t}|\right).
\end{equation}
Observe that  $\|y_{h}\|_{L^\infty(\Omega)}$ is uniformly bounded. In fact, let  $\hat{y}\in H_0^1(\Omega)\cap L^\infty(\Omega)$ be the solution to \eqref{eq:weak_st_eq} with $u=u_{h}$. Since $y_{h}$ corresponds to the Galerkin approximation of $\hat{y}$ within $\mathbb{V}_h$, the error estimate \eqref{eq:global_estimate_state_II} and Theorem \ref{thm:well_posedness_semilinear} reveal that
\begin{equation}\label{eq:aux_estimate_iii}
\|{y}_{h}\|_{L^\infty(\Omega)}
\leq 
\|\hat{y} - {y}_{h}\|_{L^\infty(\Omega)} + \|\hat{y}\|_{L^\infty(\Omega)} 
\lesssim
(h^{2-\frac{d}{2}} +1) \|u_{h} - a(\cdot,0)\|_{L^2(\Omega)}.
\end{equation}
The assumption on $u_{h}$ thus yields $\|{y}_{h}\|_{L^\infty(\Omega)} \leq C +\| a(\cdot,0)\|_{L^2(\Omega)}$. Replace  this bound into \eqref{eq:aux_estimate_vi}, and the obtained estimate into the one derived for $\|p-p_h\|_{L^2(\Omega)}$ to obtain
\begin{equation}\label{eq:aux_estimate_new}
\|p-p_h\|_{L^2(\Omega)} 
\lesssim
\|p-\hat{p}\|_{L^2(\Omega)} + h^{2-\frac{d}{2}}.
\end{equation}

We now bound $\|p-\hat{p}\|_{L^2(\Omega)}$ in \eqref{eq:aux_estimate_new}. To accomplish this task, we introduce 
\begin{multline*}
\phi:= p - \hat{p}\in W_0^{1,r}(\Omega):
\quad
(\nabla w,\nabla\phi)_{L^2(\Omega)} + \left(\tfrac{\partial a}{\partial y}(\cdot,y)\phi,w\right)_{L^2(\Omega)} \\
= \sum_{t\in\mathcal{D}}\langle (y(t)-{y}_{h}(t) )\delta_{t},w\rangle + \left(\left[\tfrac{\partial a}{\partial y}(\cdot,y_{h})-\tfrac{\partial a}{\partial y}(\cdot,y)\right]\hat{p},w\right)_{L^2(\Omega)}
\quad \forall w\in W_0^{1,r'}(\Omega).
\end{multline*}
An inf-sup condition, which follows from \cite[Theorem 1]{MR812624}, yields the estimate 
\begin{align}\label{eq:stab_phi}
\|\nabla \phi \|_{L^r(\Omega)} 
&\lesssim
\sum_{t\in\mathcal{D}}|y(t)-{y}_{h}(t) | 
+ \left\|\left[\tfrac{\partial a}{\partial y}(\cdot,y_{h})-\tfrac{\partial a}{\partial y}(\cdot,y)\right]\hat{p}\right\|_{L^2(\Omega)} 
\\
&\lesssim \|y-{y}_{h}\|_{L^\infty(\Omega)}(1+\|\hat{p}\|_{L^2(\Omega)}),
\nonumber
\end{align}
upon utilizing that $\partial a/\partial y$ is locally Lipschitz in $y$ \FF{and \eqref{eq:aux_estimate_iii}} . Since $W_{0}^{1,r}(\Omega)\hookrightarrow L^2(\Omega)$ for $r\in [2d/(d+2),d/(d-1))$, a stability estimate for the problem that $\hat{p}$ solves yields
\begin{equation*}
\|\hat{p}\|_{L^2(\Omega)}
\lesssim
\|\nabla \hat{p}\|_{L^r(\Omega)}
\lesssim
\|y_h\|_{L^\infty(\Omega)}+\sum_{t\in\mathcal{D}}|y_{t}|.
\end{equation*}
Replace the estimate for $\|y_h\|_{L^\infty(\Omega)}$ obtained in \eqref{eq:aux_estimate_iii} into this bound and the obtained one into \eqref{eq:stab_phi} to obtain
$
\|\nabla \phi \|_{L^{r}(\Omega)} \lesssim \|y-y_{h}\|_{L^\infty(\Omega)}.
$
Thus, 
\begin{equation}\label{eq:aux_estimate_iv}
\|p-\hat{p}\|_{L^2(\Omega)} = \|\phi\|_{L^2(\Omega)} 
\lesssim \| \nabla \phi \|_{L^r(\Omega)}
 \lesssim \|y-y_{h}\|_{L^\infty(\Omega)},
 \quad
 r \in \left[\tfrac{2d}{d+2},\tfrac{d}{d-1}\right),
\end{equation}
with a hidden constant that is independent of the involved continuous and discrete variables and $h$ but depends on $C$, $a$, and $\{ y_t \}_{t \in \mathcal{D}}$. It thus suffices to bound the term $\|y-y_{h}\|_{L^\infty(\Omega)}$. Invoke $\hat{y}$ and write
$
\|y-y_{h}\|_{L^\infty(\Omega)}\leq \|y-\hat{y}\|_{L^\infty(\Omega)}+\|\hat{y}-y_{h}\|_{L^\infty(\Omega)}.
$
In view of the Lipschitz property \eqref{eq:lipschitz_prop} and the estimate \eqref{eq:global_estimate_state_II}, we arrive at the bound
\begin{equation}\label{eq:aux_estimate_v}
\|y-y_{h}\|_{L^\infty(\Omega)}
\lesssim 
\|u - u_{h}\|_{L^2(\Omega)} + h^{2-\frac{d}{2}}\|u_{h} - a(\cdot,0)\|_{L^2(\Omega)}.
\end{equation}
Replacing the estimate \eqref{eq:aux_estimate_v} into \eqref{eq:aux_estimate_iv} and the obtained one into \eqref{eq:aux_estimate_new}, and taking into account the assumption on $u_{h}$, we conclude the desired estimate \eqref{eq:estimate_adj_discrete_continuous}.
\end{proof}


\subsection{Discretization of the control problem}\label{sec:fem_control_problem}
We propose two strategies to discretize \eqref{def:weak_ocp}--\eqref{eq:weak_st_eq}:  a semidiscrete approach where the control is not  discretized -- the  so-called  variational  discretization  approach  --  and  a  fully  discrete approach where the control variable is discretized with piecewise constant functions.

\subsubsection{A fully discretization scheme}
\label{sec:fully_discrete}
We propose the following fully discrete approximation of problem \eqref{def:weak_ocp}--\eqref{eq:weak_st_eq}: Find $\min J(y_{h},u_{h})$ subject to 
\begin{equation}
\label{eq:discrete_state_equation}
y_h \in \mathbb{V}_h: \quad
(\nabla y_h,\nabla v_h)_{L^2(\Omega)}+(a(\cdot,y_h),v_h)_{L^2(\Omega)}  =  (u_h,v_h)^{}_{L^2(\Omega)} \quad \forall v_{h}\in \mathbb{V}_{h},
\end{equation}
and the discrete constraints $u_{h} \in \mathbb{U}_{ad,h}$. Here, 
$
\mathbb{U}_{ad,h}:=\mathbb{U}_{h}\cap \mathbb{U}_{ad},
$
where $\mathbb{U}_{h}=\{ u_h \in L^\infty(\Omega): u_{h}|^{}_T\in \mathbb{P}_0(T) \ \forall T\in \T_{h}\}$. We recall that $\mathbb{V}_h$ is defined as in \eqref{def:piecewise_linear_set}. 

The existence of a solution follows from the compactness of $\mathbb{U}_{ad,h}$ and the continuity of $J$. Let us introduce the discrete control to state map $\mathcal{S}_{h}:  \mathbb{U}_{h} \ni u_h \mapsto y_h \in \mathbb{V}_{h}$, where $y_h$ solves \eqref{eq:discrete_state_equation}, and the reduced cost functional $j_{h}(u_{h}):=J(\mathcal{S}_{h}u_{h},u_{h})$. With these ingredients at hand, first order optimality conditions for the fully discrete optimal control problem reads as follows: If $\bar{u}_{h}$ denotes a local solution, then
\begin{equation}
\label{eq:discrete_var_ineq}
j_{h}^{\prime}(\bar{u}_{h})(u_{h}-\bar{u}_{h}) = (\bar{p}_{h}+\alpha\bar{u}_{h},u_{h}-\bar{u}_{h})_{L^2(\Omega)}  \geq  0 \quad \forall u_{h} \in \mathbb{U}_{ad,h},
\end{equation}
where $\bar{p}_{h} \in \mathbb{V}_{h}$ solves the discrete problem \eqref{eq:discrete_adjoint_equation} with $y_{h}$ replaced by $\bar{y}_{h}:=\mathcal{S}_{h}\bar{u}_{h}$.

\subsubsection{The variational discretization approach}
\label{sec:variational_discretization}

In this section, we propose a semidiscrete scheme that is based on the so-called variational discretization approach \cite{MR2122182}. This scheme discretizes only the state space (the control space is not discretized) and induces a discretization of optimal controls by projecting the optimal discrete adjoint state into the admissible control set $\mathbb{U}_{ad}$.

The scheme is defined as follows: Find $\min J(y_h,\mathfrak{u})$ subject to 
\begin{equation}
\label{eq:discrete_state_equation_var}
y_h \in \mathbb{V}_h: \quad
(\nabla y_h,\nabla v_h)_{L^2(\Omega)}+(a(\cdot,y_h),v_h)_{L^2(\Omega)}  =  (\mathfrak{u},v_h)^{}_{L^2(\Omega)} \quad \forall v_{h}\in \mathbb{V}_{h},
\end{equation}
and the control constraints $\mathfrak{u}\in \mathbb{U}_{ad}$. The existence of a discrete solution and first order optimality conditions for the semidiscrete scheme follow standard arguments. In particular, if $\bar{\mathfrak{u}}$ denotes a local solution, then
\begin{equation}
\label{eq:discrete_var_ineq_variational}
j_{h}^{\prime}(\bar{\mathfrak{u}})(u-\bar{\mathfrak{u}}) = (\bar{p}_{h}+\alpha\bar{\mathfrak{u}},u-\bar{\mathfrak{u}})_{L^2(\Omega)}  \geq  0 \quad \forall u \in \mathbb{U}_{ad},
\end{equation}
where $\bar{p}_{h} \in \mathbb{V}_{h}$ solves the discrete problem \eqref{eq:discrete_adjoint_equation} with $y_{h}$ replaced by $\bar{y}_h$. Here, $\bar{y}_h$ denotes the solution to \eqref{eq:discrete_state_equation_var} with $\mathfrak{u}$ replaced by $\bar{\mathfrak{u}}$. Notice that, in view of the variational inequality \eqref{eq:discrete_var_ineq_variational}, the following projection formula holds \cite[section 4.6]{Troltzsch}:
\begin{equation}\label{eq:projection_control_var} 
\bar{\mathfrak{u}}(x):=\Pi_{[\texttt{a},\texttt{b}]}(-\alpha^{-1}\bar{p}_h(x)) \textrm{ a.e. } x \in \Omega.
\end{equation}
Since $\bar{\mathfrak{u}}$ implicitly depends on $h$, in what follows we will adopt the notation $\bar{\mathfrak{u}}_h$.

\section{Convergence of discretizations}\label{sec:conver_of_disc}
We begin with the following error bounds, the proofs of which can be found  in \cite[Lemmas 37 and 38]{MR3586845}.

\begin{theorem}[auxiliary error estimate]\label{thm:error_estimates_state_aux}
Let $\Omega$ be a convex polytope. Assume that \textnormal{\ref{A1}}, \textnormal{\ref{A2}}, and \textnormal{\ref{A3}} hold. Let $u \in \mathbb{U}_{ad}$ and ${u}_{h} \in \mathbb{U}_{ad,h} \subset \mathbb{U}_{ad}$.
Let $y = y(u)$ be the solution to \eqref{eq:weak_st_eq} and let ${y}_{h} = y_h(u_h)$ be the solution to  \eqref{eq:discrete_state_equation}. Then, 
\begin{equation*}
\| \nabla( y-y_h)\|_{L^2(\Omega)} \lesssim h + \|u - u_h\|_{L^2(\Omega)},
\quad
\|y-y_h\|_{L^\infty(\Omega)} \lesssim h^{2-\frac{d}{2}} + \|u - u_h\|_{L^2(\Omega)}.
\end{equation*}
In addition, if $u_h\rightharpoonup u$ in $L^s(\Omega)$ as $h \downarrow 0$, with $s>d/2$, then $y_{h}\rightarrow y$ in $H_0^1(\Omega)\cap C(\bar{\Omega})$ as $h \downarrow 0$ and $j(u) \leq \liminf_{h \downarrow 0} j_h(u_h)$.
\end{theorem}

\subsection{The fully discrete scheme: convergence of discretizations}

In what follows we provide a convergence result that, in essence, guarantees that a sequence of discrete global solutions $\{ \bar u_h \}_{h>0}$ contains subsequences that converges, as $h \downarrow 0$, to global solutions of problem \eqref{def:weak_ocp}--\eqref{eq:weak_st_eq}.

\begin{theorem}[convergence of global solutions]
\label{thm:convergence_discrete_sol}
Assume that \textnormal{\ref{A1}}, \textnormal{\ref{A2}}, and \textnormal{\ref{A3}} hold. Let $h>0$ and $\bar{u}_h\in\mathbb{U}_{ad,h}$ be a global solution of the fully discrete optimal control problem. Then, there exist nonrelabeled subsequences of $\{\bar{u}_{h}\}_{h>0}$ such that $\bar u_h \mathrel{\ensurestackMath{\stackon[1pt]{\rightharpoonup}{\scriptstyle\ast}}} \bar u$ in the weak$^\star$ topology of $L^\infty(\Omega)$, as $h \downarrow 0$, and $\bar u$ corresponds to a global solution of the optimal control problem \eqref{def:weak_ocp}--\eqref{eq:weak_st_eq}. 
In addition, we have
\begin{equation}\label{eq:discrete_cont_convergence}
\lim_{h\to 0}\|\bar{u}-\bar{u}_{h}\|_{L^2(\Omega)}  = 0, \qquad
\lim_{h\to 0}j_{h}(\bar{u}_{h}) = j(\bar{u}).
\end{equation}
\end{theorem}
\begin{proof}
Since, for every $h>0$, $\bar{u}_{h}\in \mathbb{U}_{ad,h}\subset \mathbb{U}_{ad}$, the sequence $\{\bar{u}_{h}\}_{h>0}$ is uniformly bounded in $L^\infty(\Omega)$. Then, there exists a nonrelabeled subsequence such that $\bar{u}_{h} \mathrel{\ensurestackMath{\stackon[1pt]{\rightharpoonup}{\scriptstyle\ast}}} \bar{u}$ in $L^\infty(\Omega)$ as $h\downarrow 0$. In what follows, we prove that $\bar{u}\in \mathbb{U}_{ad}$ is a solution to the optimal control problem \eqref{def:weak_ocp}--\eqref{eq:weak_st_eq} and that the convergence results in \eqref{eq:discrete_cont_convergence} hold.

Let $\tilde{u}\in \mathbb{U}_{ad}$ be a global solution to \eqref{def:weak_ocp}--\eqref{eq:weak_st_eq}. Let $\Pi_{L^2}:L^2(\Omega)\to \mathbb{U}_{h}$ be the orthogonal projection operator into piecewise constant functions over $\T_h$. Notice that, for $v \in L^2(\Omega)$ and $T\in\T$, we have $\Pi_{L^2}v|_T:= \int_{T} v  \mathrm{d}x/ |T|$. Define $\tilde u_h:= \Pi_{L^2}\tilde{u}\in \mathbb{U}_{ad,h}$. Since Theorem \ref{thm:extra_regul_control} guarantees that $\tilde{u}\in H^1(\Omega)$, we immediately conclude that $\|\tilde{u}-\tilde u_h\|_{L^2(\Omega)}\to 0$ as $h\downarrow 0$. We thus invoke, the global optimality of $\tilde{u}$, Theorem \ref{thm:error_estimates_state_aux}, the global optimality of $\bar{u}_{h}$, and the convergence result $\tilde{u}_{h} \rightarrow \tilde{u}$ in $L^2(\Omega)$ to obtain
\[
j(\tilde{u}) \leq j(\bar{u}) \leq \liminf_{h \downarrow 0} j_{h}(\bar{u}_{h}) \leq \limsup_{h \downarrow 0} j_{h}(\bar{u}_{h}) \leq \limsup_{h \downarrow 0} j_{h}(\tilde u_h) =  j(\tilde{u}).
\]
This proves that $\bar{u}$ is a global solution of \eqref{def:weak_ocp}--\eqref{eq:weak_st_eq} and that $\lim_{h\downarrow 0}j_{h}(\bar{u}_{h})= j(\bar{u})$.

Finally, observe that Theorem \ref{thm:error_estimates_state_aux} yields $\bar{y}_{h} \to \bar{y}$ in $C(\bar \Omega)$ and that this implies
\[
\sum_{t \in \mathcal{D}} (\bar{y}_h(t) - y_t)^2 \rightarrow \sum_{t \in \mathcal{D}} (\bar{y}(t) - y_t)^2,
\quad h \downarrow 0.
\]
Since $j_{h}(\bar{u}_{h}) \rightarrow j(\bar{u})$, we can thus conclude that
$
\| \bar u_h \|^2_{L^2(\Omega)} \rightarrow \| \bar u \|^2_{L^2(\Omega)}
$
as $h \downarrow 0$. The weak convergence $\bar u_h \rightharpoonup \bar u$ in $L^2(\Omega)$, as $h \downarrow 0$, allows us to conclude.
\end{proof}

In what follows, we prove that strict local solutions of problem \eqref{def:weak_ocp}--\eqref{eq:weak_st_eq}  can be approximated by local solutions of the fully discrete optimal control problems.

\begin{theorem}[convergence of local solutions]
\label{thm:convergence_discrete_sol_local}
Assume that \textnormal{\ref{A1}}, \textnormal{\ref{A2}}, and \textnormal{\ref{A3}} hold. Let $\bar{u}\in\mathbb{U}_{ad}$ be a strict local minimum of \eqref{def:weak_ocp}--\eqref{eq:weak_st_eq}. Then, there exists a sequence of local minima $\{\bar{u}_h\}_{h>0}$ of the fully discrete scheme satisfying \eqref{eq:discrete_cont_convergence}.
\end{theorem}
\begin{proof}
Since $\bar{u}$ is a strict local minimum of \eqref{def:weak_ocp}--\eqref{eq:weak_st_eq}, there exists $\varepsilon > 0$ such that $\min\{ j(u): u \in\mathbb{U}_{ad}\cap B_{\varepsilon}(\bar{u})\}$ admits $\bar u$ as a unique solution. Here, $B_{\varepsilon}(\bar{u}):=\{ u \in L^2(\Omega): \|\bar{u}-u\|_{L^2(\Omega)}\leq \varepsilon\}$. On the other hand, let us consider, for $h>0$, the discrete problem: Find $\min\{j_{h}(u_h): u_h\in\mathbb{U}_{ad,h}\cap B_{\varepsilon}(\bar{u})\}$. There exists $h_{\varepsilon}$ such that this problem admits a solution for $h \leq h_{\varepsilon}$. In fact, $\Pi_{L^2} \bar u \in \mathbb{U}_{ad,h}\cap B_{\varepsilon}(\bar{u})$ for $h$ sufficiently small. As a result, $ \mathbb{U}_{ad,h}\cap B_{\varepsilon}(\bar{u})$ is closed, bounded, and nonempty in $L^2(\Omega)$.

Let $h \leq h_{\varepsilon}$ and let $\bar u_h$ be a global solution of $\min\{j_{h}(u_h): u_h\in\mathbb{U}_{ad,h}\cap B_{\varepsilon}(\bar{u})\}$. We proceed as in the proof of Theorem \ref{thm:convergence_discrete_sol} to conclude the existence of a subsequence of $\{\bar{u}_h\}_{h\leq h_\varepsilon}$ such that it converges strongly in $L^2(\Omega)$ to a solution of $\min\{ j(u): u \in\mathbb{U}_{ad}\cap B_{\varepsilon}(\bar{u})\}$. Since the latter problem admits a unique solution $\bar{u}$, we must have $\bar u_h \rightarrow \bar u$ in $L^2(\Omega)$ as $h \downarrow 0$. This implies that the constraint $\bar{u}_{h} \in B_{\varepsilon}(\bar{u})$ is not active for small $h$. As a result, $\bar{u}_h$ is a local solution of the fully discrete scheme and the convergence properties in \eqref{eq:discrete_cont_convergence} hold.
\end{proof}


\section{Error estimates}\label{sec:error_estimates}


In this section, we derive error estimates for the fully and semidiscrete schemes introduced in sections \ref{sec:fully_discrete} and \ref{sec:variational_discretization}, respectively.

\subsection{Error estimates for the fully discrete scheme}

Let $ \{ \bar{u}_h \}_{h>0} \subset \mathbb{U}_{ad,h}$ be a sequence of local minima of the fully discrete control problems such that $\bar{u}_h \to \bar{u}$ in $L^2(\Omega)$, as $h \downarrow 0$, where $\bar{u}\in\mathbb{U}_{ad}$ is a local solution of \eqref{def:weak_ocp}--\eqref{eq:weak_st_eq}; see Theorems \ref{thm:convergence_discrete_sol} and \ref{thm:convergence_discrete_sol_local}. The main goal of this section is to derive an estimate for $\bar u - \bar u_h$ in $L^2(\Omega)$.

\begin{theorem}[error estimate]\label{thm:control_estimate}
Assume that \textnormal{\ref{A1}}, \textnormal{\ref{A2}}, and \textnormal{\ref{A3}} hold. Let $a(\cdot, 0)\in L^\infty(\Omega)$. If $\bar{u}\in\mathbb{U}_{ad}$ satisfies the sufficient second order optimality condition \eqref{eq:second_order_2_2}, then there exists $h_{\ddagger} > 0$ such that the following estimate holds:
\begin{equation}\label{eq:control_error_estimate}
\|\bar{u}-\bar{u}_h\|_{L^2(\Omega)}\lesssim h|\log h| 
\qquad \forall h < h_{\ddagger},
\end{equation}
where the hidden constant is independent of $h$.
\end{theorem}

We follow \cite{MR2272157,MR2350349} and proceed by contradiction. Let us assume that $\{ \bar u_h \}_{h>0}$ converges to $\bar u$ as $h \downarrow 0$ and that \eqref{eq:control_error_estimate} does not hold. We thus have, for every $k \in \mathbb{N}$, $h_k>0$ such that
$
\|\bar{u}-\bar{u}_{h_k}\|_{L^2(\Omega)} > k h_k|\log {h_k}|,
$
and thus  $\{h_{k}\}_{k\in \mathbb{N}} \subset \mathbb{R}^{+}$ such that
\begin{equation}\label{eq:estimate_contradiction}
\lim_{ h_{k}\downarrow 0}\|\bar{u}-\bar{u}_{h_k}\|^{}_{L^2(\Omega)}\to 0, 
\qquad
\lim_{h_{k} \downarrow 0}\frac{\|\bar{u}-\bar{u}_{h_{k}}\|_{L^2(\Omega)}}{h_{k}|\log {h}_{k}|}=+\infty.
\end{equation}

To prove \eqref{eq:control_error_estimate} we need some preparatory results.

\begin{lemma}[auxiliary result]\label{lemma:aux_thm_1}
Assume that \textnormal{\ref{A1}}, \textnormal{\ref{A2}}, and \textnormal{\ref{A3}} hold. Let us assume that, in addition, \eqref{eq:control_error_estimate} is false. Let $\bar{u}\in\mathbb{U}_{ad}$ satisfies the second order optimality condition \eqref{eq:second_order_2_2}. Then, there exists $h_{\dagger} > 0$ such that 
\begin{equation}\label{eq:aux_estimate}
\textgoth{C}
\|\bar{u}-\bar{u}_h\|_{L^2(\Omega)}^2\leq [j'(\bar{u}_h)-j'(\bar{u})](\bar{u}_h-\bar{u}) \quad \forall h < h_{\dagger},
\quad
\textgoth{C} = 2^{-1} \min\{\mu,\alpha\},
\end{equation}
where $\alpha$ is the regularization parameter and $\mu$ denotes the constant appearing in  \eqref{eq:second_order_equivalent}.
\end{lemma}
\begin{proof}
Since \eqref{eq:control_error_estimate} is false, there exists a sequence $\{ h_k \}_{k \in \mathbb{N}}$ such that the limits in \eqref{eq:estimate_contradiction} hold. In an attempt to simplify the exposition of the material, in what follows, we will omit the subindex $k$, i.e., we denote $u_{h_k} = u_{h}$. Observe that $h \downarrow 0$ as $k \uparrow \infty$.

Define 
$
v_h:= (\bar{u}_h-\bar{u})/\|\bar{u}_h-\bar{u}\|_{L^2(\Omega)}.
$
We assume that (up to a subsequence if necessary) $v_h \rightharpoonup v$ in $L^2(\Omega)$ as $h\downarrow 0$. In what follows, we prove that $v\in C_{\bar{u}}$, where $C_{\bar{u}}$ is defined in \eqref{def:critical_cone}. Since $\bar{u}_h\in \mathbb{U}_{ad,h}\subset \mathbb{U}_{ad}$, it is clear that $v_h$ satisfies the sign conditions in \eqref{eq:sign_cond}. The fact that $v_h \rightharpoonup v$ in $L^2(\Omega)$, as $h\downarrow 0$, implies that $v$ satisfies \eqref{eq:sign_cond} as well. To show that $v(x) = 0$ if $\bar{\mathfrak{p}}(x)\neq 0$ for a.e.~$x\in\Omega$, we introduce 
$
\bar{\mathfrak{p}}_h:= \bar{p}_h+\alpha \bar{u}_h.
$
In view of $\|\bar{u}-\bar{u}_{h}\|_{L^2(\Omega)}\to 0$, as $h\downarrow 0$, Theorem \ref{thm:error_estimates_adjoint_aux} yields $\bar{\mathfrak{p}}_h \to\bar{\mathfrak{p}}$ in $L^2(\Omega)$ as $h\downarrow 0$. We can thus obtain
\begin{multline*}
\int_\Omega \bar{\mathfrak{p}}(x)v(x) \mathrm{d}x
=\lim_{h\to 0}\int_\Omega\bar{\mathfrak{p}}_h(x)v_h(x) \mathrm{d}x
=\lim_{h\to 0}\frac{1}{\|\bar{u}_h-\bar{u}\|_{L^2(\Omega)}}
\\
\cdot \left(\int_\Omega\bar{\mathfrak{p}}_h (\Pi_{L^2}\bar{u}-\bar{u})\mathrm{d}x 
+ \int_\Omega\bar{\mathfrak{p}}_h(\bar{u}_h - \Pi_{L^2}\bar{u}) \mathrm{d}x\right)
 =: \lim_{h\to 0}\frac{\mathbf{I}+\mathbf{II}}{\|\bar{u}_h-\bar{u}\|_{L^2(\Omega)}},
\end{multline*}
where $\Pi_{L^2}$ denotes the $L^2$-orthogonal projection operator into piecewise constant functions over $\T_h$. The discrete variational inequality \eqref{eq:discrete_var_ineq} immediately yields $\mathbf{II}\leq 0$. On the other hand, $|\mathbf{I}| \leq \|\bar{\mathfrak{p}}_{h}\|_{L^2(\Omega)} \|  \bar{u} - \Pi_{L^2}\bar{u} \|_{L^2(\Omega)} \lesssim h \| \nabla \bar u \|_{L^2(\Omega)}$, upon noticing that $\|\bar{\mathfrak{p}}_{h}\|_{L^2(\Omega)}$ is uniformly bounded by problem data: $\|\bar{\mathfrak{p}}_{h}\|_{L^2(\Omega)}\leq \|\bar{\mathfrak{p}}_{h}-\bar{\mathfrak{p}}\|_{L^2(\Omega)}+\|\bar{\mathfrak{p}}\|_{L^2(\Omega)} \lesssim 1$. On the basis of \eqref{eq:estimate_contradiction}, the previous inequalities yield
\begin{equation*}
\int_\Omega \bar{\mathfrak{p}}(x)v(x) \mathrm{d}x
\lesssim
\lim_{h\to 0}\frac{h}{\|\bar{u}_h-\bar{u}\|_{L^2(\Omega)}}
\lesssim 
\lim_{h\to 0}\frac{h|\log h|}{\|\bar{u}_h-\bar{u}\|_{L^2(\Omega)}}
=0.
\end{equation*}
Since $v$ satisfies the sign condition \eqref{eq:sign_cond}, then $\bar{\mathfrak{p}}(x)v(x) \geq 0$. Therefore the previous inequality yields $\int_{\Omega} |\bar{\mathfrak{p}}(x)v(x)| \mathrm{d}x = \int_{\Omega} \bar{\mathfrak{p}}(x)v(x) \mathrm{d}x\leq 0$. Consequently, if $\bar{\mathfrak{p}}(x)\neq 0$, then $v(x) = 0$ for a.e.~$x\in\Omega$. This allows us to conclude that $v \in C_{\bar{u}}$.

We now invoke the mean value theorem to deduce that
\begin{equation}\label{eq:difference_of_j}
[j'(\bar{u}_h)-j'(\bar{u})](\bar{u}_h-\bar{u})= j''(\hat{u}_h)(\bar{u}_h-\bar{u})^2, \quad \hat{u}_h = \bar{u} + \theta_h (\bar{u}_h - \bar{u}),
\end{equation}
where $\theta_h \in (0,1)$. Let $y_{\hat{u}_h}$ be unique solution to \eqref{eq:weak_st_eq} with $u=\hat{u}_{h}$ and $p_{\hat{u}_h}$ be the unique solution to \eqref{eq:adj_eq} with $y=y_{\hat{u}_h}$. Since $\bar u_h \rightarrow \bar u$ in $L^2(\Omega)$ as $h \downarrow 0$, we have $y_{\hat{u}_h} \rightarrow \bar y$ in $H_0^1(\Omega) \cap C(\bar \Omega)$ and $p_{\hat{u}_h} \rightarrow \bar p$ in $W_{0}^{1,r}(\Omega)$ as $h \downarrow 0$. Here $r<d/(d-1)$. Similarly, $v_h \rightharpoonup v$ in $L^2(\Omega)$ implies that $z_{v_h} \rightharpoonup z_v$ in $H^2(\Omega) \cap H_0^1(\Omega)$ as $h \downarrow 0$. Hence, invoke \eqref{eq:charac_j2}, the definition of $v_h$, and the second order condition \eqref{eq:second_order_equivalent} to obtain
\begin{multline}
\lim_{h \downarrow 0}j''(\hat{u}_h)v_h^2 
 = 
\lim_{h \downarrow 0}\left(\alpha-\left(\tfrac{\partial^2 a}{\partial y^2}(\cdot,y_{\hat{u}_h})z_{v_h}^2,p_{\hat{u}_h}\right)_{L^2(\Omega)}+\sum_{t\in \mathcal{D}}z_{v_h}^2(t)\right)
= \alpha
\\
-\left(\tfrac{\partial^2 a}{\partial y^2}(\cdot,\bar{y})z_{v}^2,\bar{p}\right)_{L^2(\Omega)}
+\sum_{t\in \mathcal{D}}z_{v}^2(t)
= 
\alpha + j''(\bar{u})v^2-\alpha\|v\|_{L^2(\Omega)}^2
\geq \alpha +(\mu-\alpha) \|v\|_{L^2(\Omega)}^2.
\end{multline}
Therefore, since $\|v\|_{L^2(\Omega)}\leq 1$, we arrive at
$
\lim_{h \downarrow 0}j''(\hat{u}_h)v_{h}^2 \geq \min\{\mu,\alpha\}>0,
$
which proves the existence of $h_{\dagger} > 0$ such that
\[
j''(\hat{u}_h)v_{h}^2 \geq 2^{-1}\min\{\mu,\alpha\} \quad \forall h< h_{\dagger}.
\]
This, in light of the definition of $v_h$ and the identity \eqref{eq:difference_of_j}, allows us to conclude.
\end{proof}

\begin{lemma}[auxiliary result]\label{lemma:aux_thm_2}
Assume that \textnormal{\ref{A1}}, \textnormal{\ref{A2}}, and \textnormal{\ref{A3}} hold, and that $a(\cdot, 0)\in L^\infty(\Omega)$. Let $u_1,u_2\in \mathbb{U}_{ad}$ and $v\in L^\infty(\Omega)$. Thus, we have
\begin{align}
\label{eq:aux_estimate_2}
|{j}^{\prime}(u_1)v- j^{\prime}_h(u_1)v|
& \lesssim
h^2|\log h|^2\|v\|_{L^\infty(\Omega)},
\\
\label{eq:aux_estimate_3}
|{j}^{\prime}_h(u_1)v-{j}^{\prime}_h(u_2)v|
& \lesssim
\|u_1-u_2\|_{L^2(\Omega)}\|v\|_{L^2(\Omega)}.
\end{align}
\end{lemma}
\begin{proof}
We proceed on the basis of two steps.

\underline{Step 1.} The goal of this step is to obtain \eqref{eq:aux_estimate_2}. To accomplish this task, we begin with a basic computation which reveals that 
$
j^{\prime}(u_{1})v = (p_{u_{1}}+\alpha u_{1},v)_{L^2(\Omega)},
$
where
\begin{equation*}
 p_{u_{1}} \in W_{0}^{1,r}(\Omega):
 \quad
(\nabla w,\nabla p_{u_{1}})_{L^2(\Omega)} + \left(\tfrac{\partial a}{\partial y}(\cdot,y_{u_{1}})p_{u_{1}},w\right)_{L^2(\Omega)} = \sum_{t\in\mathcal{D}}\langle (y_{u_{1}}(t) - y_{t})\delta_{t},w\rangle
\end{equation*}
for all $w \in W_{0}^{1,r'}(\Omega)$. Here, $r\in [2d/(d+2),d/(d-1))$ and $y_{u_{1}}$ solves \eqref{eq:weak_st_eq} with $u=u_{1}$. A similar argument yields 
$
j_{h}^{\prime}(u_{1})v = (\hat{p}_{h}+\alpha u_{1},v)_{L^2(\Omega)},
$
where $\hat{p}_{h}$ is such that
\begin{equation}\label{eq:hat_p_disc}
\hat{p}_{h} \in \mathbb{V}_{h}:
\quad
(\nabla w_{h},\nabla \hat{p}_{h})_{L^2(\Omega)} + \left(\tfrac{\partial a}{\partial y}(\cdot,\hat{y}_{h})\hat{p}_{h},w_{h}\right)_{L^2(\Omega)} = \sum_{t\in\mathcal{D}}\langle (\hat{y}_{h}(t) - y_{t})\delta_{t},w_{h}\rangle
\end{equation}
for all $w_{h} \in \mathbb{V}_{h}$.  In \eqref{eq:hat_p_disc} the variable $\hat{y}_{h}\in\mathbb{V}_{h}$ corresponds to the solution to \eqref{eq:discrete_state_equation} with $u_h$ replaced by $u_1$.
Define $\hat{p}\in W_{0}^{1,r}(\Omega)$ as the unique solution to 
\begin{equation*}
(\nabla w,\nabla\hat{p})_{L^2(\Omega)} + \left(\tfrac{\partial a}{\partial y}(\cdot,\hat{y}_{h})\hat{p},w\right)_{L^2(\Omega)} = \sum_{t\in\mathcal{D}}\langle (\hat{y}_{h}(t) - y_{t})\delta_{t},w\rangle\quad \forall w \in W_{0}^{1,r'}(\Omega).
\end{equation*}
Here, $r\in [2d/(d+2),d/(d-1))$. Notice that $\hat{p}_{h}\in \mathbb{V}_{h}$ corresponds to the finite element approximation of $\hat{p}$ within $\mathbb{V}_h$. We also notice the following stability estimate for $\hat p$:
\begin{equation}\label{eq:stab_tilde_p}
\|\nabla \hat{p}\|_{L^r(\Omega)} \lesssim \sum_{t\in\mathcal{D}}|\hat{y}_{h}(t) - y_{t}|.
\end{equation} 
With all these continuous and discrete variables at hand, we can write
\begin{equation}\label{eq:tilde_hat_triang}
{j}^{\prime}(u_1)v- j^{\prime}_h(u_1)v = 
(p_{u_{1}}-\hat{p},v)_{L^2(\Omega)} + (\hat{p}-\hat{p}_{h},v)_{L^2(\Omega)}:=\mathbf{I} + \mathbf{II}.
\end{equation}

To estimate the term $\mathbf{I}$ we define $\zeta := p_{u_{1}}-\hat{p}\in W_0^{1,r}(\Omega)$ and observe that
\begin{multline*}
(\nabla w,\nabla\zeta)_{L^2(\Omega)} + \left(\tfrac{\partial a}{\partial y}(\cdot,{y}_{u_{1}})\zeta,w\right)_{L^2(\Omega)} 
=
\sum_{t\in\mathcal{D}}\langle (y_{u_{1}}(t)-\hat{y}_{h}(t) )\delta_{t},w\rangle 
+ \left(\mathfrak{g} ,w\right)_{L^2(\Omega)},
\end{multline*}
for all $w\in W_{0}^{1,r'}(\Omega)$. Here, $\mathfrak{g}:= [\tfrac{\partial a}{\partial y}(\cdot,\hat{y}_{h})-\tfrac{\partial a}{\partial y}(\cdot,{y}_{u_{1}})]\hat{p}$. An inf-sup condition, which follows from \cite[Theorem 1]{MR812624}, yields
\begin{equation}\label{eq:stab_zeta}
\|\nabla \zeta \|_{L^r(\Omega)} 
\lesssim
\sum_{t\in\mathcal{D}}|y_{u_{1}}(t)-\hat{y}_{h}(t) | + \left\|
\mathfrak{g}
\right\|_{L^2(\Omega)}.
\end{equation}

Let us concentrate on $\| \mathfrak{g}\|_{L^2(\Omega)}$. Let $\Lambda_1, \Omega_0$ be smooth domains such that $\Omega_1 \Subset \Lambda_1 \Subset \Omega_0 \Subset \Omega$ and $\mathcal{D} \subset \Omega_{1}$. 
The basic identity
$
\| \mathfrak{g}
\|_{L^2(\Omega)}^2 = 
\| \mathfrak{g}
\|_{L^2(\Lambda_{1})}^2
+
\| \mathfrak{g}
\|_{L^2(\Omega\setminus\Lambda_{1})}^2
$
combined with the error bounds derived in Theorem \ref{thm:error_estimates_state} yield
\begin{align*}
\| \mathfrak{g}
\|_{L^2(\Omega)}^2
&\lesssim 
\|\hat{p}\|_{L^2(\Lambda_{1})}^2\|\hat{y}_{h}-{y}_{u_{1}}\|_{L^\infty(\Lambda_{1})}^2
+
\|\hat{p}\|_{L^\infty(\Omega\setminus\Lambda_{1})}^2\|\hat{y}_{h}-{y}_{u_{1}}\|_{L^2(\Omega)}^2
\\
&\lesssim
h^4|\log h|^4\|\hat{p}\|_{L^2(\Omega)}^2 + h^4\|\hat{p}\|_{L^\infty(\Omega\setminus\Lambda_{1})}^2\|u_{1} - a(\cdot,0)\|_{L^2(\Omega)}^2.
\end{align*}
The Sobolev embedding $W_{0}^{1,r}(\Omega)\hookrightarrow L^2(\Omega)$ and the fact that $u_{1}\in\mathbb{U}_{ad}$ reveal that
\begin{equation}\label{eq:combined estimate_i}
\| \mathfrak{g}
\|_{L^2(\Omega)}^2
\lesssim 
h^4|\log h|^4\|\nabla \hat{p}\|_{L^r(\Omega)}^2 + h^4\|\hat{p}\|_{L^\infty(\Omega\setminus\Lambda_{1})}^2,
\end{equation}
where the hidden constant is independent of the involved continuous and discrete variables but depends on the continuous optimal control problem data. We now invoke Theorem \ref{thm:reg_adj_state_local} to conclude that $\|\hat{p}\|_{L^\infty(\Omega\setminus\Lambda_{1})}$ is uniformly bounded. On the other hand, the stability estimate \eqref{eq:stab_tilde_p} and analogous arguments to the ones that lead to \eqref{eq:aux_estimate_iii} allows us to conclude that $\|\nabla \hat{p}\|_{L^r(\Omega)} \leq C$, where $C$ depends on $\{ y_t \}_{t \in \mathcal{D}}$, $a$, $\texttt{a}$, and $\texttt{b}$. We thus invoke \eqref{eq:combined estimate_i} to arrive at
$
\| \mathfrak{g}
\|_{L^2(\Omega)}
\lesssim h^2|\log h|^{2}.
$
This bound, estimate \eqref{eq:stab_zeta}, and the local estimate \eqref{eq:local_estimate_state}, yield
\begin{equation}\label{eq:estimate_I_zeta}
\mathbf{I} 
\leq
\|\zeta\|_{L^2(\Omega)}\|v\|_{L^2(\Omega)} 
\lesssim 
\|\nabla \zeta\|_{L^r(\Omega)}\|v\|_{L^2(\Omega)}
\lesssim
h^2|\log h|^{2}\|v\|_{L^2(\Omega)}.
\end{equation}

The control of $\mathbf{II}$ in \eqref{eq:tilde_hat_triang} follows immediately from the error estimate \eqref{eq:error_estimate_fundamental}:
\begin{equation}\label{eq:estimate_II_tilde_discrete}
\mathbf{II} \leq \| \hat{p} - \hat{p}_{h}\|_{L^1(\Omega)}\|v\|_{L^\infty(\Omega)}
\lesssim
h^2|\log h|^2\|v\|_{L^\infty(\Omega)}.
\end{equation}

Upon combining \eqref{eq:tilde_hat_triang}, \eqref{eq:estimate_I_zeta}, and \eqref{eq:estimate_II_tilde_discrete}, we obtain the estimate \eqref{eq:aux_estimate_2}.

\underline{Step 2.}  We now obtain \eqref{eq:aux_estimate_3}. From the previous step, we have that
$
j_{h}^{\prime}(u_{1})v = (\hat{p}_{h}+\alpha u_{1},v)_{L^2(\Omega)},
$
where $\hat{p}_{h}\in \mathbb{V}_{h}$ solves \eqref{eq:hat_p_disc}. On the other hand, similar arguments yield
$
j_{h}^{\prime}(u_{2})v = (\tilde{p}_{h}+\alpha u_{2},v)_{L^2(\Omega)},
$
where $\tilde{p}_{h}\in \mathbb{V}_{h}$ is the unique solution to the problem
\begin{equation*}
(\nabla w_{h},\nabla \tilde{p}_{h})_{L^2(\Omega)} + \left(\tfrac{\partial a}{\partial y}(\cdot,\tilde{y}_{h})\tilde{p}_{h},w_{h}\right)_{L^2(\Omega)} = \sum_{t\in\mathcal{D}}\langle (\tilde{y}_{h}(t) - y_{t})\delta_{t},w_{h}\rangle\quad \forall w_{h} \in \mathbb{V}_{h}.
\end{equation*}
Here, $\tilde{y}_{h}\in\mathbb{V}_{h}$ corresponds to the solution to \eqref{eq:discrete_state_equation} with $u_h$ replaced by $u_2$. Therefore, 
\begin{equation}\label{eq:estimate_II_iii}
|j_{h}^{\prime}(u_{1})v - j_{h}^{\prime}(u_{2})v|
\leq 
\left(\|\hat{p}_{h}-\tilde{p}_{h}\|_{L^2(\Omega)}+\alpha\| u_{1}-u_{2}\|_{L^2(\Omega)}\right)\|v\|_{L^2(\Omega)}.
\end{equation}

It thus suffices to bound $\|\hat{p}_{h}-\tilde{p}_{h}\|_{L^2(\Omega)}$. To accomplish this task, we define
\begin{multline*}
\xi \in W_{0}^{1,r}(\Omega):
\quad
(\nabla w,\nabla\xi)_{L^2(\Omega)} + \left(\tfrac{\partial a}{\partial y}(\cdot,\hat{y}_{h})\xi,w\right)_{L^2(\Omega)}  \\
=
\sum_{t\in\mathcal{D}}\langle (\hat{y}_{h}(t)-\tilde{y}_{h}(t) )\delta_{t},w\rangle + \left(\left[\tfrac{\partial a}{\partial y}(\cdot,\tilde{y}_{h})-\tfrac{\partial a}{\partial y}(\cdot,\hat{y}_{h})\right]\tilde{p}_{h},w\right)_{L^2(\Omega)} \quad \forall w\in W_{0}^{1,r'}(\Omega).
\end{multline*} 
We also define $\xi_{h}:=\hat{p}_{h}-\tilde{p}_{h}\in\mathbb{V}_{h}$ and immediately observe that $\xi_{h}$ corresponds to the finite element approximation of $\xi$ within $\mathbb{V}_h$.
We thus invoke basic estimates, \eqref{eq:error_estimate_fundamental_II}, and a stability estimate for the problem that $\xi$ solves to arrive at
\begin{align*}
\|\xi_{h}\|_{L^2(\Omega)} 
&
\leq 
\|\xi-\xi_{h}\|_{L^2(\Omega)}+\|\xi\|_{L^2(\Omega)}
\lesssim \|\xi-\xi_{h}\|_{L^2(\Omega)}+\|\nabla\xi\|_{L^r(\Omega)}
\\
&\lesssim (h^{2-\frac{d}{2}}+1)\left(\|\hat{y}_{h}-\tilde{y}_{h}\|_{L^\infty(\Omega)} + \left\|\left[\tfrac{\partial a}{\partial y}(\cdot,\tilde{y}_{h})-\tfrac{\partial a}{\partial y}(\cdot,\hat{y}_{h})\right]\tilde{p}_{h}\right\|_{L^2(\Omega)}\right).
\end{align*}
This estimate, in light of assumption \ref{A3}, immediately yields
\begin{equation}\label{eq:estimate_II_iv}
\|\hat{p}_{h}-\tilde{p}_{h}\|_{L^2(\Omega)} = \|\xi_{h}\|_{L^2(\Omega)}
\lesssim
(1+\|\tilde{p}_{h}\|_{L^2(\Omega)})\|\hat{y}_{h}-\tilde{y}_{h}\|_{L^\infty(\Omega)}.
\end{equation}

We now bound $\|\hat{y}_{h}-\tilde{y}_{h}\|_{L^\infty(\Omega)}$. Let us first recall that $y_{u_{i}}$ solves \eqref{eq:weak_st_eq} with $u=u_{i}$, where $i\in\{1,2\}$, and that $\hat{y}_{h}$ and $\tilde{y}_{h}$ correspond to the finite element approximations of $y_{u_{1}}$ and $y_{u_{2}}$, respectively. Since $a(\cdot, \hat{y}_{h})-a(\cdot,\tilde{y}_{h})=\tfrac{\partial a}{\partial y}(\cdot, \mathsf{y}_{h})(\hat{y}_{h}-\tilde{y}_{h})$, where $\mathsf{y}_{h} = \tilde{y}_{h}+\theta_{h}(\hat{y}_{h}-\tilde{y}_{h})$ and $\theta_{h}\in (0,1)$, 
we deduce that $\hat{y}_{h}-\tilde{y}_{h}\in\mathbb{V}_{h}$ solves the problem
\[
(\nabla (\hat{y}_{h}-\tilde{y}_{h}),\nabla v_{h})_{L^2(\Omega)}+\left(\tfrac{\partial a}{\partial y}(\cdot, \mathsf{y}_{h})(\hat{y}_{h}-\tilde{y}_{h}), v_{h}\right)_{L^2(\Omega)}= (u_{1}-u_{2},v_{h})_{L^2(\Omega)} \quad \forall v_{h}\in\mathbb{V}_{h}.
\]
Define $\eta\in H_0^1(\Omega)$ as the solution to
\[
(\nabla \eta,\nabla v)_{L^2(\Omega)}+\left(\tfrac{\partial a}{\partial y}(\cdot, \mathsf{y}_{h})\eta, v\right)_{L^2(\Omega)}= (u_{1}-u_{2},v)_{L^2(\Omega)} \quad \forall v\in H_0^1(\Omega).
\]
Noticing that $\hat{y}_{h}-\tilde{y}_{h}\in\mathbb{V}_{h}$ corresponds to the finite element approximation of $\eta$ within $\mathbb{V}_h$, we conclude, in view of estimate \eqref{eq:global_estimate_state_II}, that
\begin{equation*}
\|\hat{y}_{h}-\tilde{y}_{h}\|_{L^\infty(\Omega)}
\leq
\|(\hat{y}_{h}-\tilde{y}_{h})-\eta\|_{L^\infty(\Omega)}+\|\eta\|_{L^\infty(\Omega)}
\lesssim
(h^{2-\frac{d}{2}}+1)\|u_{1}-u_{2}\|_{L^2(\Omega)}.
\end{equation*}
Replace this bound into \eqref{eq:estimate_II_iv} and the obtained one into \eqref{eq:estimate_II_iii} to conclude 
\begin{equation}\label{eq:estimate_II_v}
|j_{h}^{\prime}(u_{1})v - j_{h}^{\prime}(u_{2})v|
\lesssim
(1+\|\tilde{p}_{h}\|_{L^2(\Omega)}+\alpha)\|u_{1}-u_{2}\|_{L^2(\Omega)}\|v\|_{L^2(\Omega)}.
\end{equation}
We observe that similar arguments to the ones used to derive \eqref{eq:aux_estimate_iii} yield
$
\|\tilde{p}_{h}\|_{L^2(\Omega)}\lesssim \|u_{2}-a(\cdot,0)\|_{L^2(\Omega)}+\sum_{t\in\mathcal{D}}|y_{t}|\lesssim C,
$
where $C>0$. This concludes the proof.
\end{proof}  

Inspired by \cite[Lemma 7.5]{MR2272157}, \cite[Lemma 4.17]{MR2338434}, and \cite[Lemma 43]{MR3586845}, we prove the existence of a suitable auxiliary variable and derive an error estimate.

\begin{lemma}[error estimate for an auxiliary variable]\label{lemma:aux_var}
There exists $h_{\star}>0$ such that for $h < h_{\star}$ there exists $u_h^*\in \mathbb{U}_{ad,h}$ satisfying $j'(\bar{u})( \bar u - u_h^{*}) = 0$ and 
\[
\|\bar{u}-u_h^*\|_{L^2(\Omega)}\leq Ch, \qquad C > 0.
\]
\end{lemma}
\begin{proof}
Define, for each $T\in \T_h$,
$
I_{T}:=\int_{T} \bar{\mathfrak{p}}(x) \mathrm{d}x
$
and $u_h^{*} \in \mathbb{U}_{h}$ by
\begin{equation}\label{eq:definition_u*}
u_h^{*}|^{}_{T}:=
\displaystyle\frac{1}{I_{T}}\int_{T} \bar{\mathfrak{p}}(x)\bar{u}(x) \mathrm{d}x \text{ if } I_{T}\neq 0,
\qquad 
u_h^{*}|^{}_{T}:=
\displaystyle\frac{1}{|T|}\int_{T} \bar{u}(x) \mathrm{d}x  \text{ if } I_{T}= 0.
\end{equation}
We recall that $\bar{\mathfrak{p}}= \bar{p}+\alpha\bar{u}$. In view of the fact that $\bar{u} \in C^{0,1}(\bar \Omega)$, which follows from Theorem \ref{thm:extra_regul_control}, there exists $h_{\star} >0$ such that
$
|\bar{u}(x_{1})-\bar{u}(x_{2})|\leq (\texttt{b}-\texttt{a})/2 
$
for every $h < h_{\star}$ and $x_{1},x_{2} \in T$. This implies, in particular, that, for each $T \in \T_h$, $\bar{u}$ do not take both values $\texttt{a}$ and $\texttt{b}$ in $T$. Therefore, with \eqref{eq:derivative_j} at hand, we deduce that either, for a.e~$x\in T$, $\bar{\mathfrak{p}}(x) \geq 0$ or $\bar{\mathfrak{p}}(x)\leq 0$. Consequently, we have that $I_T=0$ if and only if $\bar{\mathfrak{p}}(x)=0$ for a.e.~$x\in T$, and that, if $I_T\neq 0$, $\bar{\mathfrak{p}}(x)/I_{T}\geq 0$ for a.e.~$x\in T$. From this fact, and in view of the generalized mean value theorem, we conclude the existence of $x_T\in T$ such that $u_h^{*}|^{}_T=\bar{u}(x_T)$. Since $u_h^{*}\in\mathbb{U}_h$, we have thus obtained that $u_h^{*}\in\mathbb{U}_{ad,h}$. Now, let $T\in\T_{h}$. We estimate $\| \bar u - u_h^{*} \|_{L^2(T)}$ as follows:
\[
\| \bar u - u_h^{*} \|_{L^2(T)} \leq \| \bar u - \Pi_{L^2} \bar u \|_{L^2(T)} + \| \Pi_{L^2} \bar u  - u_h^{*} \|_{L^2(T)} \lesssim h \| \nabla \bar u \|_{L^2(T)} + h^{\frac{d}{2}} \| \bar u \|_{L^{\infty}(T)}.
\]
We finally observe that \eqref{eq:definition_u*} immediately yields
\begin{equation*}
j'(\bar{u})u_{h}^{*} = (\bar{\mathfrak{p}},u_{h}^{*})_{L^2(\Omega)}=\sum_{T\in \T_{h}}(\bar{\mathfrak{p}},u_{h}^{*})_{L^2(T)}=\sum_{T\in\T_{h}}(\bar{\mathfrak{p}},\bar{u})_{L^2(T)}=j'(\bar{u})\bar{u}.
\end{equation*}
This concludes the proof.
\end{proof}

\emph{Proof of Theorem \ref{thm:control_estimate}.} Adding and subtracting the term ${j}^{\prime}_h(\bar{u}_h)(\bar{u}-\bar{u}_h)$ in the right hand side of inequality \eqref{eq:aux_estimate} yield
\begin{equation}
\|\bar{u}-\bar{u}_h\|_{L^2(\Omega)}^2 
\lesssim 
[j'(\bar{u})-{j}^{\prime}_h(\bar{u}_h)](\bar{u}-\bar{u}_h) +[{j}^{\prime}_h(\bar{u}_h)- j'(\bar{u}_h)](\bar{u}-\bar{u}_h)\quad \forall h<h_{\dagger}.
\label{eq:aux_aux}
\end{equation}
Invoke inequality \eqref{eq:aux_estimate_2} in conjunction with that fact that $\bar{u},\bar{u}_{h}\in\mathbb{U}_{ad}$ to arrive at
\begin{equation}\label{eq:error_ct_1}
\|\bar{u}-\bar{u}_h\|_{L^2(\Omega)}^2 
\leq
C \left( [j'(\bar{u})-{j}^{\prime}_h(\bar{u}_h)](\bar{u}-\bar{u}_h) + h^2|\log h|^2 \right),
\end{equation}
where $C>0$. We now estimate $[j'(\bar{u})-{j}^{\prime}_h(\bar{u}_h)](\bar{u}-\bar{u}_h)$. To accomplish this task, we set $u=\bar{u}_h$ in \eqref{eq:variational_inequality} and $u_h=u_h^*$ in \eqref{eq:discrete_var_ineq} to obtain
\begin{equation*}
0 \leq j'(\bar{u})(\bar{u}_h-\bar{u}), \qquad
0 \leq {j}^{\prime}_h(\bar{u}_h)(u_h^*-\bar{u}_h) ={j}^{\prime}_h(\bar{u}_h)(u_h^*-\bar{u})+{j}^{\prime}_h(\bar{u}_h)(\bar{u}-\bar{u}_h).
\end{equation*}
Adding these inequalities we arrive at
$
[j'(\bar{u}) - {j}^{\prime}_h(\bar{u}_h)](\bar{u}-\bar{u}_h)
\leq 
{j}^{\prime}_h(\bar{u}_h)(u_h^*-\bar{u}).
$
We utilize that $u_h^*$ is such that $j'(\bar{u})(u_h^*-\bar{u})=0$, which follows from Lemma \ref{lemma:aux_var}, to obtain
\begin{multline}\label{eq:error_ct_1-2}
[j'(\bar{u}) - {j}^{\prime}_h(\bar{u}_h)](\bar{u}-\bar{u}_h)
\leq 
[{j}^{\prime}_h(\bar{u}_h)-j'(\bar{u})](u_h^*-\bar{u})\\ 
 =
[{j}^{\prime}_h(\bar{u}_h)-{j}^{\prime}_h(\bar{u})](u_h^*-\bar{u})+[{j}^{\prime}_h(\bar{u})-j'(\bar{u})](u_h^*-\bar{u}).
\end{multline}
We thus apply estimates \eqref{eq:aux_estimate_2} and \eqref{eq:aux_estimate_3} to obtain
\begin{equation*}
[j'(\bar{u}) - {j}^{\prime}_h(\bar{u}_h)](\bar{u}-\bar{u}_h)
\lesssim 
\|\bar{u}_h-\bar{u}\|_{L^2(\Omega)}\|u_h^*-\bar{u}\|_{L^2(\Omega)}+h^2|\log h|^2.
\end{equation*}
Invoke Young's inequality and the estimate of Lemma \ref{lemma:aux_var} to arrive at
\begin{equation}\label{eq:error_ct_2}
[j'(\bar{u}) - {j}^{\prime}_h(\bar{u}_h)](\bar{u}-\bar{u}_h)
\leq 
\tfrac{1}{2C}\|\bar{u}_h-\bar{u}\|_{L^2(\Omega)}^2+\tilde{C}h^2( 1 + |\log h|^2)
\end{equation}
for every $h < h_{\star}$. Here, $\tilde{C} >0$. Finally, replacing estimate \eqref{eq:error_ct_2} into \eqref{eq:error_ct_1} we conclude \eqref{eq:control_error_estimate}. This, which contradicts \eqref{eq:estimate_contradiction}, concludes the proof. \hfill $\proofbox$

We now improve, in two dimensions, the error estimate of Theorem \ref{thm:control_estimate}.

\begin{theorem}[improved error estimate]\label{thm:control_estimate_improved}
Let $d=2$. In the framework of Theorem \ref{thm:control_estimate}, we have the following optimal error estimate:
\begin{equation}\label{eq:control_error_estimate_improved}
\|\bar{u}-\bar{u}_h\|_{L^2(\Omega)}\lesssim h \qquad \forall h < \tilde{h}_{\ddagger},
\end{equation}
where the hidden constant is independent of $h$.
\end{theorem}
\begin{proof}
We proceed, again, by contradiction and assume that $\{ \bar u_h \}_{h>0}$ converges to $\bar u$ as $h \downarrow 0$ and \eqref{eq:control_error_estimate_improved} does not hold. Hence, we can find, for every $k \in \mathbb{N}$, $h_k>0$ such that
$\|\bar{u}-\bar{u}_{h_k}\|_{L^2(\Omega)} > k h_k$, and thus a sequence $\{h_{k}\}_{k\in \mathbb{N}} \subset \mathbb{R}^{+}$ such that
\begin{equation}\label{eq:estimate_contradiction_improved}
\lim_{ h_{k}\downarrow 0}\|\bar{u}-\bar{u}_{h_k}\|^{}_{L^2(\Omega)}\to 0, 
\qquad
\lim_{h_{k} \downarrow 0} h_{k}^{-1}\|\bar{u}-\bar{u}_{h_{k}}\|_{L^2(\Omega)}=+\infty.
\end{equation}
Withing this setting, the proof of estimate \eqref{eq:aux_estimate} follows verbatim the arguments used in the proof of Lemma \ref{lemma:aux_thm_1} upon replacing \eqref{eq:estimate_contradiction} by \eqref{eq:estimate_contradiction_improved}. The next ingredient is
\begin{equation}\label{eq:aux_estimate_2_improved}
|{j}^{\prime}(u_1)v- j^{\prime}_h(u_1)v|
\lesssim
h\|v\|_{L^2(\Omega)},
\qquad
u_{1}\in\mathbb{U}_{ad}, v\in L^2(\Omega).
\end{equation}
To obtain \eqref{eq:aux_estimate_2_improved}, we follow the arguments in Lemma \ref{lemma:aux_thm_2}: it suffices to bound \EO{the terms} $\mathbf{I}$ and $\mathbf{II}$; $\mathbf{I}$ being already controlled. \EO{Observe that, since we are in two dimensions, the error estimate \eqref{eq:error_estimate_fundamental_II} yields error bound} $\mathbf{II} \leq \| \hat{p} - \hat{p}_{h}\|_{L^2(\Omega)}\|v\|_{L^2(\Omega)} \lesssim h\|v\|_{L^2(\Omega)}$. 

To obtain \eqref{eq:control_error_estimate_improved} we thus proceed as in the proof of Theorem \ref{thm:control_estimate}. Invoke \eqref{eq:aux_aux} and \eqref{eq:aux_estimate_2_improved} to obtain
$\|\bar{u}-\bar{u}_h\|_{L^2(\Omega)}^2 \lesssim (j'(\bar{u})-{j}^{\prime}_h(\bar{u}_h))(\bar{u}-\bar{u}_h) + h\|\bar{u}-\bar{u}_h\|_{L^2(\Omega)}$. We now utilize \eqref{eq:error_ct_1-2}, \eqref{eq:aux_estimate_2_improved}, and \eqref{eq:aux_estimate_3} to derive the error estimate
\begin{equation*}
\|\bar{u}-\bar{u}_h\|_{L^2(\Omega)}^2  \lesssim  (\|\bar{u}_h-\bar{u}\|_{L^2(\Omega)}+h)\|u_h^*-\bar{u}\|_{L^2(\Omega)}+ h\|\bar{u}-\bar{u}_h\|_{L^2(\Omega)}.
\end{equation*}
In view of Lemma \ref{lemma:aux_var}, and Young's inequality, we arrive at $ \|\bar{u}-\bar{u}_h\|_{L^2(\Omega)}  \lesssim  h$ for every $h$ sufficiently small. This bound contradicts \eqref{eq:estimate_contradiction_improved} and concludes the proof.
\end{proof}

\begin{remark}[$d=3$] \EO{We notice that, the use of estimate \eqref{eq:error_estimate_fundamental_II} in three dimensions leads to the deteriorated bound $\mathbf{II} \leq \| \hat{p} - \hat{p}_{h}\|_{L^2(\Omega)}\|v\|_{L^2(\Omega)} \lesssim h^{\frac{1}{2}}\|v\|_{L^2(\Omega)}$.}
\end{remark}

\subsection{Error estimates for the variational discretization approach}
Let $\{ \bar{\mathfrak{u}}_h \}_{h>0} \subset \mathbb{U}_{ad}$ be a sequence of local minima of the semidiscrete optimal control problems such that $\bar{\mathfrak{u}}_h \to \bar{u}$ in $L^2(\Omega)$, as $h \downarrow 0$, where $\bar{u}\in\mathbb{U}_{ad}$ is a local solution of \eqref{def:weak_ocp}--\eqref{eq:weak_st_eq}; see Theorems \ref{thm:convergence_discrete_sol} and \ref{thm:convergence_discrete_sol_local}. In what follows, we derive an error estimate for $\| \bar u - \bar{\mathfrak{u}}_h \|_{L^2(\Omega)}$.

The following result is instrumental.

\begin{lemma}[auxiliary result]\label{lemma:aux_thm_1_var}
Assume that \textnormal{\ref{A1}}, \textnormal{\ref{A2}}, and \textnormal{\ref{A3}} hold. Let $\bar{u}\in\mathbb{U}_{ad}$ be such that \eqref{eq:second_order_2_2} holds. Then, there exists $h_{\dagger} > 0$ such that 
\begin{equation}\label{eq:aux_estimate_var}
\textgoth{C}
\|\bar{u}-\bar{\mathfrak{u}}_h\|_{L^2(\Omega)}^2\leq [j'(\bar{\mathfrak{u}}_h)-j'(\bar{u})](\bar{\mathfrak{u}}_h-\bar{u}) \quad \forall h < h_{\dagger},
\quad
\textgoth{C} = 2^{-1} \min\{\mu,\alpha\},
\end{equation}
where $\alpha$ is the regularization parameter and $\mu$ denotes the constant appearing in \eqref{eq:second_order_equivalent}.
\end{lemma}
\begin{proof}
We adapt the proof of Lemma \ref{lemma:aux_var} and define  $v_h:= (\bar{\mathfrak{u}}_h-\bar{u})/\|\bar{\mathfrak{u}}_h-\bar{u}\|_{L^2(\Omega)}$. We assume that (up to a subsequence if necessary) $v_h \rightharpoonup v$ in $L^2(\Omega)$ as $h\downarrow 0$. The arguments in the proof of Lemma \ref{lemma:aux_var} immediately yields that $v$ satisfies \eqref{eq:sign_cond}. We now show that $v(x) = 0$ if $\bar{\mathfrak{p}}(x)\neq 0$ for a.e.~$x\in\Omega$. Invoke Theorem \ref{thm:error_estimates_adjoint_aux} to conclude that $\|\bar{\mathfrak{u}}_h-\bar{u}\|_{L^2(\Omega)}\to 0$ implies that $\bar{p}_h \to\bar{p}$ in $L^2(\Omega)$ as $h\downarrow 0$. This convergence and the variational inequality \eqref{eq:discrete_var_ineq_variational} with $u=\bar{u}$, allow us to arrive at
\begin{equation}
\int_\Omega \bar{\mathfrak{p}}(x)v(x) \mathrm{d}x
=
\frac{1}{\|\bar{\mathfrak{u}}_h-\bar{u}\|_{L^2(\Omega)}}\lim_{h\to 0}\int_\Omega(\bar{p}_h+\alpha\bar{\mathfrak{u}}_h)(\bar{\mathfrak{u}}_h-\bar{u}) \mathrm{d}x
\leq 
0.
\label{eq:aux_int}
\end{equation}
Since $v$ satisfies \eqref{eq:sign_cond}, we obtain $\bar{\mathfrak{p}}(x)v(x) \geq 0$. Therefore, \eqref{eq:aux_int} allows us to conclude that $\int_{\Omega} |\bar{\mathfrak{p}}(x)v(x)| \mathrm{d}x = \int_{\Omega} \bar{\mathfrak{p}}(x)v(x) \mathrm{d}x\leq 0$. Consequently, if $\bar{\mathfrak{p}}(x)\neq 0$, then $v(x) = 0$ for a.e.~$x\in\Omega$. This allows us to conclude that $v \in C_{\bar{u}}$.

The rest of the proof follows the same arguments developed in the proof of Lemma \ref{lemma:aux_var}. For the sake of brevity, we skip details.
\end{proof}

\begin{theorem}[error estimate]\label{thm:control_estimate_var1}
Assume that \textnormal{\ref{A1}}, \textnormal{\ref{A2}}, and \textnormal{\ref{A3}} hold, and that $a(\cdot, 0)\in L^\infty(\Omega)$. Let $\bar{u}\in\mathbb{U}_{ad}$ be such that \eqref{eq:second_order_2_2} holds. Then,
\begin{equation}\label{eq:control_error_estimate_var1}
\|\bar{u}-\bar{\mathfrak{u}}_h\|_{L^2(\Omega)} \lesssim h|\log h| \qquad \forall h < h_{\ddagger},
\end{equation}
with a hidden constant that is independent of $h$. 
\end{theorem}
\begin{proof}
Utilize \eqref{eq:aux_estimate_var}, \eqref{eq:variational_inequality}, and \eqref{eq:discrete_var_ineq_variational}, to deduce that
\begin{equation}\label{eq:ineq_var_error}
\|\bar{u}-\bar{\mathfrak{u}}_h\|_{L^2(\Omega)}^2
\lesssim
[j'(\bar{\mathfrak{u}}_h)-j_h'(\bar{\mathfrak{u}}_h)](\bar{\mathfrak{u}}_h-\bar{u}).
\end{equation}
Hence, in view of \eqref{eq:aux_estimate_2}, we immediately conclude that $\|\bar{u}-\bar{\mathfrak{u}}_h\|_{L^2(\Omega)}\lesssim h|\log h|$.
\end{proof}

\begin{lemma}[$\bar{u}(x)=\FF{\bar{\mathfrak{u}}_h}(x)$ on $\cup_{t\in\FF{\mathcal{E}}}B_{t}$]\label{lemma:uh_u}
Let $d=2$. For any $M>0$ there exists $h_{M}$ such that $|\bar{p}_h(x)|>M$ for every $x\in \cup_{t\in\FF{\mathcal{E}}}C_{t}$ and $h < h_{M}$. Here, \EO{$\mathcal{E}$ is defined in \eqref{eq:set_not_zero}} and $C_{t} \subset\Omega$ denotes a suitable open ball centered at $t\in\FF{\mathcal{E}}$ of strictly positive radius. This, in particular, implies that $\bar{u}(x)=\FF{\bar{\mathfrak{u}}_h}(x)$ for every $x\in \cup_{t\in\FF{\mathcal{E}}}B_{t}$; $B_{t}\subset\Omega$ being defined similarly.
\end{lemma}
\begin{proof}
\EO{Let $t\in \mathcal{E}$. We invoke the Lipschitz property \eqref{eq:lipschitz_prop} and the error estimates \eqref{eq:control_error_estimate_var1} and \eqref{eq:global_estimate_state_II} to arrive at
\begin{equation*}
0 
< |\bar{y}(t)-y_t| 
\leq |\bar{y}(t)-\hat{y}(t)| + |\hat{y}(t)-\bar{y}_h(t)| + |\bar{y}_h(t)-y_t|
\lesssim h|\log h| + |\bar{y}_h(t)-y_t|,
\end{equation*}
where $\hat{y}\in H_0^1(\Omega)$ denotes the unique solution to \eqref{eq:weak_st_eq} with $u$ replaced by $\bar{\mathfrak{u}}_h$. We can thus conclude the existence of $h_{\vee}>0$ such that, for $h < h_{\vee}$, we have $\bar{y}_h(t) - y_t \neq 0$. An adaption of the arguments elaborated in the proofs of \cite[Lemma 5.7]{MR3973329} and \cite[Theorem 4.6]{MR3614014}, that basically entail replacing the bilinear form $(\nabla \cdot, \nabla \cdot)_{L^2(\Omega)}$ by 
\begin{equation*}
\mathcal{B}(\cdot,\cdot):=(\nabla \cdot, \nabla \cdot)_{L^2(\Omega)} + \left(\tfrac{\partial a}{\partial y}(x,\bar{y}_h)\cdot,\cdot\right)_{L^2(\Omega)},
\end{equation*}
combined with the asymptotic behavior of $\hat{p}\in W^{1,r}(\Omega)$ -- the unique solution to \eqref{eq:hat_p} with $y_h$ replaced by $\bar{y}_h$ -- yield the existence of $h_{M}>0$ and suitable open balls $C_{t}$, with $t \in \mathcal{E}$, such that $|\bar{p}_h(x)| > M$ for every $x\in \cup_{t\in \mathcal{E}}C_{t}$ and $h < h_{M}$. On the basis of this bound, the projection formula \eqref{eq:projection_control_var} combined with the arguments elaborated at the end of the proof of Theorem \ref{thm:extra_regul_control} yield the desired result.}
\end{proof}

\begin{theorem}[improved error estimate]\label{thm:control_estimate_var}
Let $d=2$. \EO{Assume that $\mathcal{E}=\mathcal{D}$}. In the framework of Theorem \ref{thm:control_estimate}, we have the following \EO{quasi-}optimal error estimate:
\begin{equation}\label{eq:control_error_estimate_var}
\|\bar{u}-\bar{\mathfrak{u}}_h\|_{L^2(\Omega)} \lesssim h^2|\log h|^2 \qquad \forall h < h_{\ddagger},
\end{equation}
with a hidden constant that is independent of $h$. 
\end{theorem}
\begin{proof}
We begin with the following error estimate:
\begin{equation}\label{eq:aux_estimate_2_improved_var}
|[{j}^{\prime}(\bar{\mathfrak{u}}_h)- j^{\prime}_h(\bar{\mathfrak{u}}_h)](\bar{u}-\bar{\mathfrak{u}}_h)|
\lesssim
h^2|\log h|^2\|\bar{u}-\bar{\mathfrak{u}}_h\|_{L^2(\Omega)}.
\end{equation}
To obtain \eqref{eq:aux_estimate_2_improved_var}, we follow the arguments in Lemma \ref{lemma:aux_thm_2}. It suffices to bound $\mathbf{I}$ and $\mathbf{II}$; $\mathbf{I}$ being already controlled. To bound $\mathbf{II}$, we utilize Lemma \ref{lemma:uh_u} and \eqref{eq:error_estimate_fundamental_III}:
\begin{equation*}
\mathbf{II} 
\leq
\| \hat{p} - \hat{p}_{h}\|_{L^2(\Omega\setminus\cup_{t\in\mathcal{D}}\bar{B}_{t})}\|\bar{u}-\bar{\mathfrak{u}}_h\|_{L^2(\Omega\setminus\cup_{t\in\mathcal{D}}\bar{B}_{t})} 
\lesssim h^2|\log h|^2\|\bar{u}-\bar{\mathfrak{u}}_h\|_{L^2(\Omega)}.
\end{equation*} 

In view of \eqref{eq:ineq_var_error} and \eqref{eq:aux_estimate_2_improved_var}, we immediately arrive at the desired estimate. 
\end{proof}

\EO{
\begin{remark}[improved error estimate when $\mathcal{E} \subsetneq \mathcal{D}$] 
We notice that, if $\mathcal{E} \subsetneq \mathcal{D}$, the results of Lemma \ref{lemma:uh_u} allow us to control the term $\mathbf{II}$, within the proof of Theorem \ref{thm:control_estimate_var}, in a slightly different way:
\[
\mathbf{II} 
\leq
\| \hat{p} - \hat{p}_{h}\|_{L^2(\Omega\setminus\cup_{t\in\mathcal{E}}\bar{B}_{t})}\|\bar{u}-\bar{\mathfrak{u}}_h\|_{L^2(\Omega\setminus\cup_{t\in\mathcal{E}}\bar{B}_{t})}.
\]
We would like now to apply \eqref{eq:error_estimate_fundamental_III}. Unfortunately, estimate \eqref{eq:error_estimate_fundamental_III} yields a control on $\Omega\setminus\cup_{t\in\mathcal{D}}\bar{B}_{t}$ and $\Omega\setminus\cup_{t\in\mathcal{D}}\bar{B}_{t} \subsetneq \Omega\setminus\cup_{t\in\mathcal{E}}\bar{B}_{t}$. We thus proceed as follows:
\begin{align*}
\mathbf{II} 
& \leq (\| \hat{p} - \hat{p}_{h}\|_{L^2(\Omega\setminus\cup_{t\in\mathcal{D}}\bar{B}_{t})} + \| \hat{p} - \hat{p}_{h}\|_{L^2(\cup_{t\in \mathcal{D}\setminus\mathcal{E}}\bar{B}_{t})})\|\bar{u}-\bar{\mathfrak{u}}_h\|_{L^2(\Omega)}\\
& \lesssim (h^2|\log h|^2 + \| \hat{p} - \hat{p}_{h}\|_{L^2(\cup_{t\in \mathcal{D}\setminus\mathcal{E}}\bar{B}_{t})})\|\bar{u}-\bar{\mathfrak{u}}_h\|_{L^2(\Omega)},
\end{align*} 
where we have used estimate \eqref{eq:error_estimate_fundamental_III}. To control $e_{p}:= \hat{p} - \hat{p}_{h}$ on $\cup_{t\in \mathcal{D}\setminus\mathcal{E}}\bar{B}_{t}$, we invoke similar arguments to the ones that yield \eqref{eq:error_estimate_fundamental_III}, the fact that $y_{t}=\bar{y}(t)$ for $t\in\mathcal{D}\setminus\mathcal{E}$, the Lipschitz property \eqref{eq:lipschitz_prop}, and the error estimate \eqref{eq:global_estimate_state_II}. These arguments yield
\begin{multline*}
\| e_{p} \|^2_{L^2(\Sigma)} 
\lesssim h^2|\log h|^2\| e_{p} \|_{L^2(\Sigma)}  
+ h \| e_{p} \|_{L^2(\Sigma)} 
\left(  \|\bar{y}_h - \hat{y}\|_{L^\infty(\Omega)} + \|\hat{y}-\bar{y}\|_{L^\infty(\Omega)} \right)
\\
\lesssim \left( h^2|\log h|^2 + h^2 + h\|\bar{\mathfrak{u}}_h -\bar{u} \|_{L^2(\Omega)} \right) \| e_{p} \|_{L^2(\Sigma)},
\end{multline*}
where $\Sigma = \cup_{t\in \mathcal{D}\setminus\mathcal{E}}B_{t}$. Here, $\hat{y}\in H_0^1(\Omega)$ denotes the unique solution to \eqref{eq:weak_st_eq} with $u$ replaced by $\bar{\mathfrak{u}}_h$. In view of \eqref{eq:control_error_estimate_var1} we can thus deduce that $\mathbf{II}\lesssim h^2|\log h|^2 \|\bar{u}-\bar{\mathfrak{u}}_h\|_{L^2(\Omega)}$.
\end{remark}
}

\section{Numerical example}\label{sec:num_ex}

In this section, we conduct a numerical experiment that illustrates the performance of the fully and semidiscrete schemes of sections \ref{sec:fully_discrete} and \ref{sec:variational_discretization}, respectively, when are used to approximate the solution to \eqref{def:cost_func}--\eqref{def:box_constraints}. The numerical experiment has been carried out with the help of a code that was implemented using \texttt{C++}. All matrices have been assembled exactly and global linear systems were solved using the multifrontal massively parallel sparse direct solver (MUMPS) \cite{MUMPS2}. The right hand sides and the approximation errors were computed by a quadrature formula which is exact for polynomials of degree nineteen $(19)$.

For a given mesh $\mathscr{T}_{h}$, we seek $(\bar{y}_{h},\bar{p}_{h},\bar{u}_{h})\in \mathbb{V}_{h}\times \mathbb{V}_{h} \times \mathbb{U}_{ad,h}$ and $(\bar{y}_{h},\bar{p}_{h},\bar{\mathfrak{u}}_{h})\in \mathbb{V}_{h}\times \mathbb{V}_{h} \times \mathbb{U}_{ad}$ that solve the fully and semidiscrete problems, respectively. We solve the fully discrete scheme on the basis of a primal--dual active set strategy \cite[section 2.12.4]{Troltzsch} combined with a fixed point algorithm. The semidiscrete problem is solved by using a semi--smooth Newton method.

\textbf{The numerical example.} We set $\Omega=(0,1)^{2}$, $a(\cdot,y)=y^3$,  $\texttt{b}=-\texttt{a}=5$, $\alpha=0.1$, 
\[
\mathcal{D}=\{(0.25,0.25), (0.75,0.25),(0.75,0.75),(0.25,0.75) \},
\]
$y_{(0.25,0.25)} = 3$, $y_{(0.75,0.25)} = -3$, $y_{(0.75,0.75)} = 3$, and $y_{(0.25,0.75)} = -3$.

In the absence of an exact solution, we calculate the error committed in the approximation of the optimal control variable by taking as a reference solution the discrete optimal control obtained on a fine triangulation $\T_{h}$: the mesh $\T_{h}$ is such that $h \approx 10^{-3}$. In Figure \ref{fig:ex}, we observe that optimal experimental rates of convergence, in terms of approximation, are attained: $\mathcal{O}(h)$ for the fully discrete scheme and $\mathcal{O}(h^2)$ for the semidiscrete scheme. These rates are in agreement with Theorems \ref{thm:control_estimate_improved} and \ref{thm:control_estimate_var}.

\begin{figure}[!ht]
\centering
\includegraphics[trim={0 0 0 0},clip,width=3.6cm,height=3.5cm,scale=0.33]{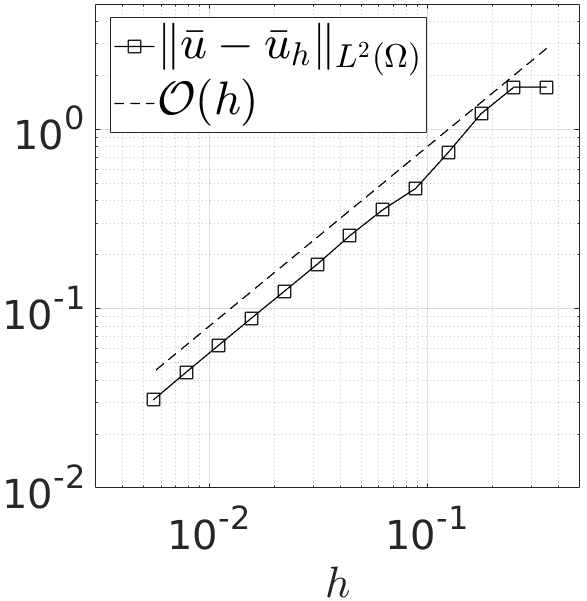}
\qquad
\includegraphics[trim={0 0 0 0},clip,width=3.6cm,height=3.5cm,scale=0.33]{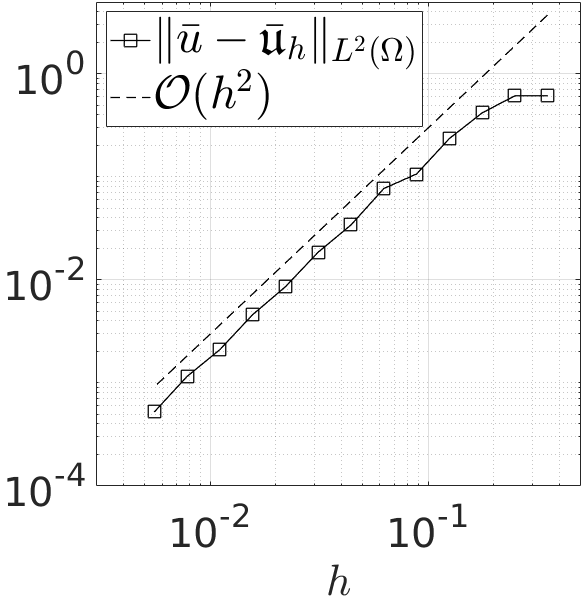}
\caption{Experimental rates of convergence for $\|\bar{u}-\bar{u}_{h}\|_{L^2(\Omega)}$ and $\|\bar{u}-\bar{\mathfrak{u}}_{h}\|_{L^2(\Omega)}$.}
\label{fig:ex}
\end{figure}

\textbf{Acknowledgments.} We would like to thank the anonymous referees for several comments and suggestions that led to better results and an improved presentation. In particular, we would like to thank a reviewer for providing a proof for the improved result presented in Theorem \ref{thm:control_estimate_improved}.

\bibliographystyle{siamplain}
\bibliography{bi_tracking}

\end{document}